\documentclass{amsart}

\usepackage{capt-of}
\usepackage[skip=6pt,font=footnotesize]{caption}
\usepackage{amsmath}
\usepackage{amsfonts}
\usepackage{amssymb}
\usepackage{graphicx}
\usepackage{algorithmic}
\usepackage{todonotes}
\usepackage{xcolor}
\usepackage{array}
\usepackage{extarrows}
\usepackage{wasysym}
\usepackage{tikz}
\usepackage{pgfplots}
\usetikzlibrary{plotmarks}
\usepackage{pgfplotstable}
\pgfplotsset{compat=1.12}
\usetikzlibrary{external}

\numberwithin{equation}{section}

\usepackage{hyperref}
\hypersetup{colorlinks=true,allcolors=blue}

\usepackage[capitalize,nameinlink]{cleveref}
\crefname{section}{section}{sections}
\crefname{subsection}{subsection}{subsections}
\Crefname{section}{Section}{Sections}
\Crefname{subsection}{Subsection}{Subsections}
\crefname{algo}{Algorithm}{Algorithms}
\crefname{table}{Table}{Tables}
\crefformat{equation}{\textup{#2(#1)#3}}
\crefrangeformat{equation}{\textup{#3(#1)#4--#5(#2)#6}}
\crefmultiformat{equation}{\textup{#2(#1)#3}}{ and \textup{#2(#1)#3}}
{, \textup{#2(#1)#3}}{, and \textup{#2(#1)#3}}
\crefrangemultiformat{equation}{\textup{#3(#1)#4--#5(#2)#6}}%
{ and \textup{#3(#1)#4--#5(#2)#6}}{, \textup{#3(#1)#4--#5(#2)#6}}{, and \textup{#3(#1)#4--#5(#2)#6}}
\Crefformat{equation}{#2Equation~\textup{(#1)}#3}
\Crefrangeformat{equation}{Equations~\textup{#3(#1)#4--#5(#2)#6}}
\Crefmultiformat{equation}{Equations~\textup{#2(#1)#3}}{ and \textup{#2(#1)#3}}
{, \textup{#2(#1)#3}}{, and \textup{#2(#1)#3}}
\Crefrangemultiformat{equation}{Equations~\textup{#3(#1)#4--#5(#2)#6}}%
{ and \textup{#3(#1)#4--#5(#2)#6}}{, \textup{#3(#1)#4--#5(#2)#6}}{, and \textup{#3(#1)#4--#5(#2)#6}}
\crefdefaultlabelformat{#2\textup{#1}#3}

\usepackage{xr}

\newtheorem{theorem}{Theorem}[section]
\newtheorem{definition}[theorem]{Definition}
\newtheorem{proposition}[theorem]{Proposition}
\newtheorem{lemma}[theorem]{Lemma}
\newtheorem{remark}[theorem]{Remark}

\newcommand{\divergence}{\operatorname{div}} 
\newcommand{\supp}{\operatorname{supp}}

\newcommand{\spann}{\operatorname{span}}   
\newcommand{\dist}{\operatorname{dist}}
\newcommand{\diam}{\operatorname{diam}}
\newcommand{\proj}{\operatorname{proj}}

\newcommand{\la}{\langle} 
\newcommand{\ra}{\rangle}
\newcommand{\lb}{\lbrace} 
\newcommand{\rb}{\rbrace}
\newcommand{\R}{\mathbb{R}} 
\newcommand{\N}{\mathbb{N}}
\newcommand{\HH}{\mathcal{H}} 
\newcommand{\BB}{\mathcal{B}} 
\newcommand{\longharpoonup}{\overset{l\rightarrow\infty}{\relbar\hspace{-2pt}\relbar\hspace{-2pt}\relbar\hspace{-2pt}\rightharpoonup}}
\newcommand{\lenergy}{\left\vert\mspace{-1.1mu}\left\vert\mspace{-1.1mu}\left\vert}
\newcommand{\renergy}{\right\vert\mspace{-1.1mu}\right\vert\mspace{-1.1mu}\right\vert}
\newcommand{\energy}{\vert\mspace{-1.1mu}\vert\mspace{-1.1mu}\vert}

\newcommand{\epsalgo}{\varepsilon_{\mathrm{algofail}}}
\newcommand{\epsfail}{\varepsilon_{\mathrm{fail}}}

\newcommand{\tol}{\normalfont{\texttt{tol}}}
\newcommand{\tolGFEM}{\tol_\mathrm{GFEM}}

\ifpdf
\hypersetup{
	pdftitle={Optimal local approximation spaces for parabolic problems},
	pdfauthor={J. Schleu{\ss} and K. Smetana}
}
\fi

\title{Optimal local approximation spaces for parabolic problems}

\author{Julia Schleu{\ss}}
\address{Institute for Computational and Applied Mathematics, University of M\"unster, Einsteinstr. 62, 48149 M\"unster, Germany, julia.schleuss@uni-muenster.de.}
	
\author{Kathrin Smetana}
\address{Department of Mathematical Sciences, Stevens Institute of Technology, 1 Castle Point Terrace, Hoboken, NJ 07030, United States of America, ksmetana@stevens.edu.}

\date{\today}

\thanks{The work of Julia Schleu{\ss} was funded by the Deutsche Forschungsgemeinschaft (DFG, German Research Foundation) under Germany's Excellence Strategy EXC 2044-390685587, Mathematics M\"unster: Dynamics-Geometry-Structure.}



\begin{document}
	
	\begin{abstract}
	We propose local space-time approximation spaces for parabolic problems that are optimal in the sense of Kolmogorov and may be employed in multiscale and domain decomposition methods. The diffusion coefficient can be arbitrarily rough in space and time. To construct local approximation spaces we consider a compact transfer operator that acts on the space of local solutions and covers the full time dimension. The optimal local spaces are then given by the left singular vectors of the transfer operator. To prove compactness of the latter we combine a suitable parabolic Caccioppoli inequality with the compactness theorem of Aubin-Lions. In contrast to the elliptic setting [I. Babu\v{s}ka and R. Lipton, \textit{Multiscale Model. Simul.}, 9 (2011), pp. 373-406] we need an additional regularity result to combine the two results. Furthermore, we employ the generalized finite element method to couple local spaces and construct an approximation of the global solution. Since our approach yields reduced space-time bases, the computation of the global approximation does not require a time stepping method and is thus computationally efficient. Moreover, we derive rigorous local and global a priori error bounds. In detail, we bound the global approximation error in a graph norm by the local errors in the $L^2(H^1)$-norm, noting that the space the transfer operator maps to is equipped with this norm. Numerical experiments demonstrate an exponential decay of the singular values of the transfer operator and the local and global approximation errors for problems with high contrast or multiscale structure regarding space and time.
	\end{abstract}

	\maketitle


\section{Introduction}

In certain industrial applications and environmental sciences, modeling and simulating phenomena such as the transport of pollutants in the groundwater is of great interest. These problems exhibit highly varying and heterogeneous multiscale features since both local fine-scale effects such as capillary pressures and the coarse-scale flow behavior have a significant influence on the overall concentration and distribution of the pollutant. In addition, the respective coefficient functions can be rough in space and time. Therefore, a numerical simulation using standard techniques such as the finite element (FE) or the finite volume method can be prohibitively expensive. In fact, the FE method based on classical polynomials can perform arbitrarily bad for heterogeneous problems with rough coefficients \cite{BO00}. Well-known strategies to address these problems are domain decomposition \cite{G06,HK18,KW02,SD14} and multiscale methods \cite{BL11,WE03,EG13,FH98,GG12,HW97,HS96,MP14,MZ07,OZ11,OZ17,OZ14}. The latter are based on ansatz functions which incorporate the local behavior of the (numerical) solution of the PDE.

Localizable multiscale methods that allow controlling both the error due to localization and the global approximation error at a certain rate, but do not rely on structural assumptions such as periodicity or scale separation, have been developed in the last decade. So far, such type of multiscale methods have been proposed for problems with coefficients that are rough in space and include optimal local approximation spaces \cite{BL11,CE18,EG13}, localized bases based on the flux transfer property \cite{OZ11}, adaptive local finite elements (AL basis) \cite{GG12}, the local orthogonal decomposition \cite{MP18,MP14}, rough polyharmonic splines \cite{OZ14}, and gamblets \cite{OZ17}.

In this paper, we propose optimal local approximation spaces, and, to the best of our knowledge, for the first time such type of multiscale methods for parabolic problems with coefficients that are rough in both space and time. We conjecture that the concepts developed in this paper and the corresponding numerical analysis are relevant for all these methods to address coefficients that are rough in space and time.

To construct ansatz functions that incorporate the local behavior of the global solution, we consider the space of all local solutions of the PDE on a target subdomain. To localize the computations, we introduce a strictly larger oversampling subdomain and then restrict the local solutions to the target subdomain. The key observation motivating our approach is the very rapid, exponential decay of the solutions from the boundary of the oversampling domain to the inner target domain which reveals that the solution space of the PDE is locally low-dimensional. To detect the functions that still persist on the target subdomain and are thus relevant for approximation, we use a transfer operator that maps boundary data in space and time on the boundary of the oversampling domain to the respective solution on the target subdomain over the whole time interval. Compactness of the transfer operator facilitates its singular value decomposition via the Hilbert-Schmidt theorem \cite[Theorem 8.94]{RR04} and thus enables the approximation of its range using only few modes if the singular values decay fast. As the exponential decay of the solutions in the interior yields a fast decay of energy, the singular values indeed decay fast. Spanning the local space by the leading left singular vectors of the transfer operator results in an approximation space that is optimal in the sense of Kolmogorov \cite{K36} and hence minimizes the approximation error among all spaces of the same dimension. To prove compactness of the transfer operator, we combine a parabolic Caccioppoli inequality with the compact embedding of suitable Sobolev spaces in $L^2(L^2)$. The Caccioppoli inequality represents the exponential decay behavior of higher frequencies in an analytic fashion and allows to bound the $L^2(H^1)$-norm of local solutions on the target subdomain in terms of their $L^2(L^2)$-norm on the oversampling subdomain. It can therefore be seen as an inverse Poincaré inequality. In contrast to the elliptic setting \cite{BL11} the regularities do not match a priori and we therefore exploit an additional regularity result for the weak solutions to combine the two results.

To construct an approximation of the global solution, we employ the generalized finite element method (GFEM) \cite{BC94,BM96} as one example for coupling the local spaces since it allows to bound the global approximation error in terms of the local error contributions. In contrast to existing approaches \cite{GS00} we propose, to the best of our knowledge, for the first time a space-time GFEM based on local space-time ansatz functions, as for certain problems a reduction only with respect to the spatial variable can become expensive if the time discretization involves many time steps. Such problems comprise, for instance, multiscale diffusion coefficients that are varying non-periodically in time. In those cases either the reduced spatial bases would become very large since snapshots for many time points have to be included or one would have to use an adaptive-in-time procedure based on smaller time intervals. The computation of the global approximation for the space-time GFEM we propose here does not require a time stepping method. As several numerical experiments demonstrate a very rapid and exponential decay of the approximation errors for an increasing number of basis functions, the solution of the global system is computationally very efficient.

As one key contribution of this paper we prove a rigorous a priori error bound for the local approximation error in the $L^2(H^1)$-seminorm. We highlight that the proofs in the elliptic setting crucially rely on the fact that the solution of the PDE also minimizes an energy functional, which is not true for the solution of the parabolic PDE. We show that as a consequence additional data terms have to be included in the a priori error bound and that the bound generally does not hold without these extra terms. Moreover, as one major contribution of this paper, we prove that the global approximation error in a suitable graph norm can be controlled only by the local errors in the $L^2(H^1)$-seminorm. The key argument to additionally control the time derivative of the global error in a certain reduced dual norm is a (Petrov-)Galerkin orthogonality of the approximation error and the reduced test space. Exploiting the global a priori error result we propose an adaptive algorithm for the localized construction of the local ansatz spaces such that the global GFEM approximation satisfies a prescribed global error tolerance. Finally, as another contribution of this paper, we also show how to deal with non-homogeneous boundary conditions.

Localizable multiscale methods for parabolic problems with coefficients that are rough in space can, for instance, be found in \cite{CE18,MP18,OZ11,OZ17,OZ14}. 

The local orthogonal decomposition (LOD) has been introduced in \cite{MP14} for elliptic multiscale problems and generalized to parabolic multiscale problems with highly varying spatial diffusion coefficients in \cite{MP18}. The key idea of the LOD is to express the space $H^1_0$ as a direct sum of a fine-scale space, which is the kernel of an $H^1$-stable interpolation operator on a coarse mesh, and a multiscale space that is defined as the difference of the coarse finite element space and its orthogonal projection onto the fine-scale space. In this way, the decomposition is orthogonal with respect to the energy inner product. Exploiting the Caccioppoli inequality an exponential decay of the basis functions is shown \cite{MP14} and consequently the ansatz functions can be approximated on local subdomains. In the parabolic setting the multiscale basis functions are combined with a backward Euler time stepping scheme.

In \cite{OZ14} rough polyharmonic splines are introduced as the solutions of constrained minimization problems that have a built-in decay behavior. This justifies their approximation by localized interpolation functions that are computed on local subdomains. The resulting approximation error is bounded using a certain Caccioppoli inequality. Concerning parabolic problems an implicit time discretization is proposed.

A probabilistic methodology motivated by game theory is introduced in \cite{O17}. The so-called gamblets are locally computed in a hierarchic fine-to-coarse fashion, decay exponentially, and induce an orthogonal multiresolution decomposition of the solution space. The approach is generalized to parabolic (and hyperbolic) problems in \cite{OZ17}, where an implicit Euler time discretization is used. 

In \cite{OZ11} localized bases for (elliptic, hyperbolic, and) parabolic problems with spatial $L^\infty$-diffusion coefficients are introduced. The approach is based on the flux transfer property and compactness properties due to source terms of sufficient regularity. In particular, local spatial approximation spaces are constructed by solving elliptic PDEs on local subdomains and an implicit time discretization is used.

Similar to our approach, the methods discussed above are based on compactness properties of certain operators and exploit Caccioppoli-type inequalities. We emphasize that, in contrast to existing approaches, our approach is not restricted to parabolic problems with diffusion coefficients that only vary rapidly in space, but is able to deal with coefficients that are arbitrarily rough in both space and time. Furthermore, we do not construct reduced spaces only with respect to the spatial variable and consequently no time stepping procedure is required to compute the global approximation.

In \cite{BH14,BL11} optimal local approximation spaces for elliptic PDEs with rough coefficient functions are introduced via a compact restriction operator that acts on the space of local solutions. The local spaces are then coupled using the GFEM \cite{BC94,BM96}. Furthermore, optimal interface spaces for (parametrized) elliptic problems are introduced in \cite{SP16}, generalized to geometry changes in \cite{S20}, and also proposed in \cite{CH20}. Concerning (real-world) applications the optimal local approximation spaces are employed, for instance, for the construction of digital twins \cite{KK20,KK18} and in the context of data assimilation \cite{TP18}. Other options for approximating the optimal local reduced spaces besides the random sampling technique \cite{BS18} employed here, are proposed in \cite{BL20,KL20}.

Finally, there has recently been a growing number of contributions concerning localized model order reduction for parametrized problems \cite{AH12,EP13,HK13,IQ12,MR02,MR04,MR15}. Local ansatz spaces are generated either via snapshots that are precomputed on local reference domains and reused for geometrically similar subdomains \cite{EP13,HK13,IQ12,MR02,MR04}, a combination of greedy-type reduced basis (RB) approximations and liftings of (eigenfunction or snapshot) interface modes \cite{EP13,HK13,MR15}, or greedy RB approximations with a principal component analysis compression \cite{AH12}. We refer to \cite{BI19} for an overview on localized model order reduction procedures for parametrized problems. In the time-dependent setting localized approaches addressing flow simulations can, for instance, be found in \cite{FI18,GG16,KF15}. Here, the local ansatz spaces are build using proper orthogonal decomposition \cite{FI18,GG16}, discrete empirical interpolation \cite{GG16}, or greedy-type RB approximations \cite{KF15}.

The remainder of this paper is organized as follows. In \cref{model_problem} we introduce the parabolic model problem. Subsequently, the main contributions of this paper are developed in \cref{local_appr_spaces,global_approximation}. We propose optimal local approximation spaces in \cref{local_appr_spaces} and discuss their computational realization in \cref{impl_and_random}. Moreover, in \cref{global_approximation} we address the construction of a global approximation via GFEM and provide local and global a priori error bounds. Finally, we present numerical experiments in \cref{numerical_experiments} to demonstrate the approximation properties of our local and global reduced spaces and draw some conclusions in \cref{conclusions}.


\section{Model problem: the linear heat equation}
\label{model_problem}

In this section we introduce the linear heat equation as a representative model problem
for parabolic problems. To that end, let $\Omega \subseteq \R^n$ denote a large, bounded Lipschitz domain of dimension $n\in\lb1,2,3\rb$ with $\partial \Omega = \Sigma_D \cup \Sigma_N$, where $\Sigma_D$ denotes the Dirichlet and $\Sigma_N$ the Neumann boundary, respectively. Furthermore, let $I=(0,T)\subseteq \R$ denote a time interval for an arbitrary $0<T<\infty$. We consider the following initial boundary value problem for the linear heat equation: Find the temperature $u: I\times\Omega \rightarrow\R$ such that
\begin{align} \label{heateq}
	\begin{array}{rlll}
		u_t(t,x)-\divergence(\alpha(t,x)\nabla u(t,x))&=&f(t,x) &\text{for every }(t,x) \in I\times\Omega,\\
		u(t,x)&=&g_D(t,x)&\text{for every }(t,x) \in I \times \Sigma_D,\\
		\alpha\nabla u(t,x) \cdot n(x)&=&g_N(t,x)&\text{for every }(t,x) \in I \times \Sigma_N,\\
		u(0,x)&=&u_0(x) &\text{for every }x \in \Omega.
	\end{array}
\end{align}
Here, $\alpha \in L^\infty(I\times\Omega)^{n \times n}$ denotes the heat conductivity coefficient that satisfies $\alpha_0(t,x)\,\vert v \vert^2 \leq v^\text{T} \alpha(t,x) \,v \leq \alpha_1(t,x)\, \vert v \vert^2$ for every $v\in\R^{n}$ and $0<\alpha_0<\alpha_0(t,x)<\alpha_1(t,x)<\alpha_1<\infty$ for almost every $(t,x)\in I\times\Omega$ and $\alpha_0,\alpha_1 \in \R $. Moreover, the function $f \in L^2(I,V^*)$ represents a heat source, $u_0\in L^2(\Omega)$ denotes the initial temperature, $g_D \in L^2(I,H^\frac{1}{2}(\Sigma_D))$ and $g_N \in L^2(I,H^{-\frac{1}{2}}(\Sigma_N))$ denote the Dirichlet and Neumann boundary data, and $n$ is the outer unit normal. The spatial test space is given by $V:=\lb w \in H^1(\Omega) \mid w = 0 \text{ on }\Sigma_D\rb$ and $V^*$ denotes its dual space, where $\Vert \cdot \Vert_V := \Vert\alpha^\frac{1}{2} \cdot \Vert_{L^2(\Omega)} +\Vert \alpha^\frac{1}{2}\nabla \cdot \Vert_{L^2(\Omega)}$. A corresponding weak formulation of \cref{heateq} then reads as follows: Find $u \in L^\infty(I,L^2(\Omega))\cap L^2(I,V_{\Sigma_D})$ such that
\begin{align}\label{global_formulation}
	\begin{split}
		&-\int_I (u(t),v)_{L^2(\Omega)}\,\varphi_t(t)\,dt\; + \int_I (\alpha \nabla u (t), \nabla v)_{L^2(\Omega)}\,\varphi(t)\,dt \\ 
		&= \int_I \la f(t), v\ra_V \,\varphi(t)\,dt+ \int_I \la g_N(t), v\ra_{H^\frac{1}{2}(\Sigma_N)}\,\varphi(t)\,dt\quad \forall\,v\in V,\;\varphi \in C_0^\infty(I)
	\end{split}
\end{align}
and it holds $u(0)=u_0$ in $L^2(\Omega)$, where $V_{\Sigma_D}:= \lb w \in H^1(\Omega) \mid w = g_D \text{ on } \Sigma_D \rb$.


\section{Optimal local approximation spaces}
\label{local_appr_spaces}
In this section we propose local space-time approximation spaces, which are optimal in the sense of Kolmogorov, for the linear heat equation with coefficients that are rough in space and time, extending the approach from \cite{BL11} for the elliptic setting. In \cref{motivation_local} we start with a motivation before we describe the construction of local ansatz spaces for local subdomains in the interior of the global domain in \cref{interior} and subdomains located at the global boundary in  \cref{boundary}.

\subsection{Motivation}
\label{motivation_local}

To tackle heterogeneous problems with rough coefficients, we propose localizable multiscale methods based on ansatz functions which incorporate the local behavior of the global solution of the PDE. To this end, we consider the space of all local solutions of the PDE with arbitrary Dirichlet boundary values on the boundary of the oversampling domain $\Omega^{out}$; see \cref{inner_domains} for an illustration. We showcase that the local solution space on the target subdomain $\Omega^{in}$ can be well approximated using only few functions via an example \cite{SP16}: Consider $-\Delta u = 0$ in $\Omega^{out}=(-2,2)\times(0,1)$ with homogeneous Neumann boundary conditions on $(-2,2)\times\lbrace 0,1\rbrace$ and arbitrary Dirichlet boundary conditions on $\lb-2,2\rb \times (0,1)$. Using separation of variables, we conclude that all solutions of this problem can be written as $u(x,y)=a_0 + b_0 x + \sum_{n=1}^\infty \cos(n\pi y) [a_n \cosh(n\pi x) + b_n \sinh(n\pi y)]$, where $a_n,b_n \in \R$ are determined by the Dirichlet boundary values for $n=0,\ldots,\infty$. We observe in \cref{figure_ex_decay} a very rapid decay of the higher frequencies of the solutions ($\cos(n\pi y)$ for higher $n$) from the boundary into the interior of the domain $\Omega^{out}$, which implies that the solution space of the PDE is locally low-dimensional. To detect the functions that still persist on $\Omega^{in}$ and are thus relevant for approximation purposes, we introduce a transfer operator $P$ whose range is the space of local solutions of the PDE on $\Omega^{in}$.

After discretization, say, with the FE method, the transfer operator can be represented by a matrix $\mathbf{P}$. It is then well-known that the range of this matrix can be optimally approximated by its $k$ leading left singular vectors and that the projection error satisfies $\Vert \mathbf{P} - \mathbf{U}_k \mathbf{U}_k^\top \mathbf{P}\Vert_2 = \sigma_{h,k+1}$ (Eckart-Young theorem e.g. in \cite{GV13}); here the columns of $\mathbf{U}_k$ contain the $k$ leading singular vectors and $\sigma_{h,k+1}$ denotes the $k+1$st singular value of $\mathbf{P}$. While the \emph{discrete} transfer operator is trivially compact because of its finite rank, in the \emph{continuous} setting we need to prove compactness of the transfer operator to facilitate its singular value decomposition via the Hilbert-Schmidt theorem  \cite[Theorem 8.94]{RR04}. Then, the space $\Lambda_k$ spanned by the $k$ leading left singular vectors is an optimal approximation space in the sense of Kolmogorov, meaning that it minimizes the approximation error among all linear spaces of dimension $k$. In addition, we have as in the discrete setting $\Vert P - \text{proj}_{\Lambda_k} P \Vert = \sigma_{k+1}$, where $\text{proj}_{\Lambda_k}$ denotes the orthogonal projection onto $\Lambda_k$ and $\sigma_{k+1}$ is the $k+1$st singular value of $P$. Thanks to the fast decay of the singular values, related to the exponential decay of the solutions in the interior, very few left singular vectors suffice for an accurate approximation.

\begin{figure}
	\begin{minipage}{0.49\textwidth}
		\centering 
		\includegraphics[height=1.8cm]{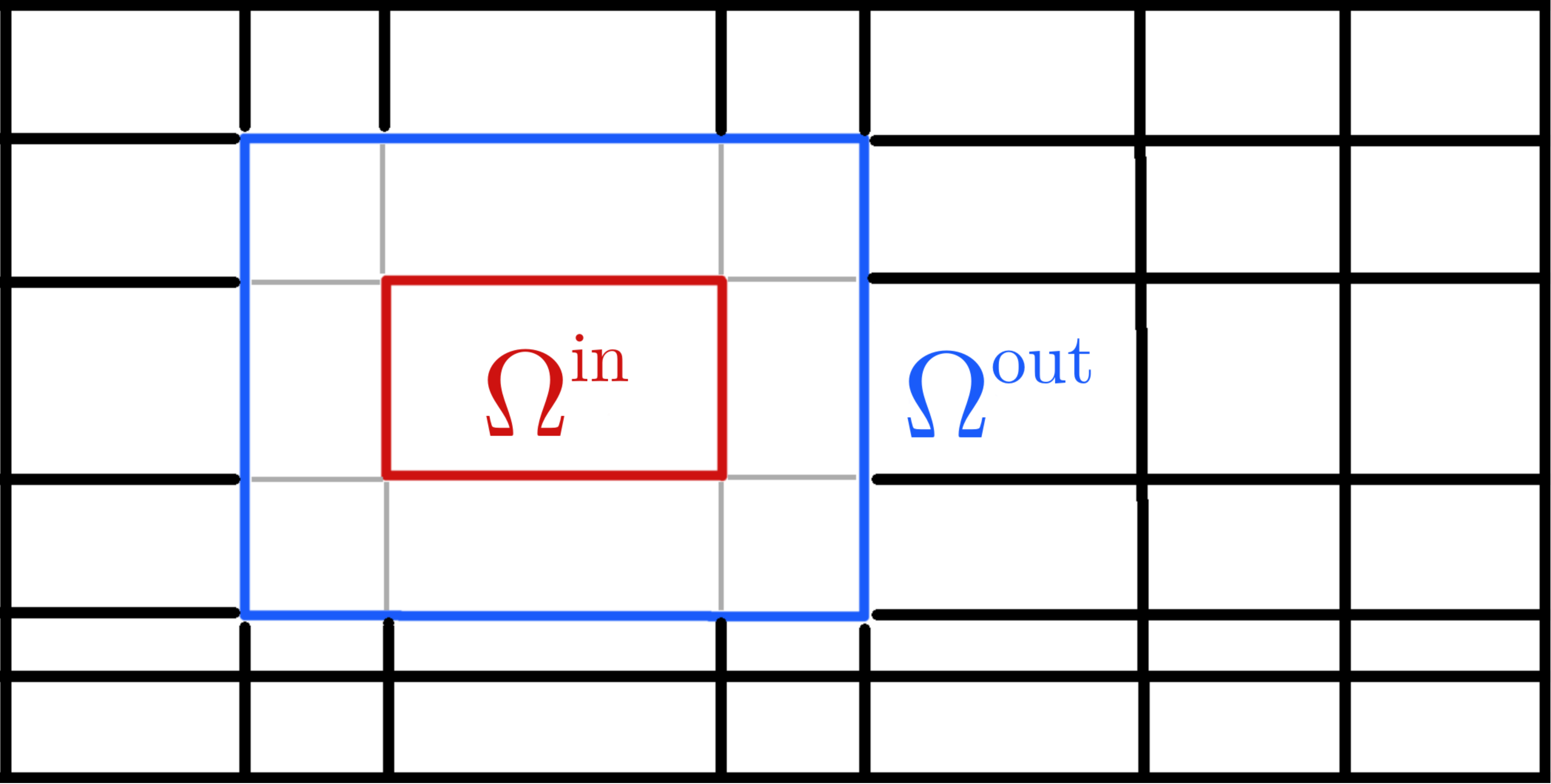}
		\captionsetup{width=0.95\linewidth}
		\captionof{figure}{\footnotesize Local target domain $\Omega^{in}$ and oversampling domain $\Omega^{out}$ in the interior of $\Omega$.}
		\label{inner_domains}
	\end{minipage}
	\hfill
	\begin{minipage}{0.49\textwidth}
		\centering
		\begin{tikzpicture}
		\begin{axis}[
		width=5.5cm,
		height=3.16cm,
		xmin=-2,
		xmax=2,
		ymin=-0.02,
		ymax=0.5,
		xtick = {-2,-1,0,1,2},
		ytick={0,0.2,0.4},
		legend style = {font=\scriptsize,at={(1.44,1.0)}},
		tick label style={font=\scriptsize}  
		]
		\addplot+[ mark = x, black, mark indices = {13,30,50,80,135,200,265,321,351,371,388}] table[x index=0, y index=1, forget plot] {motivation_decay.dat};
		\addlegendentry{$n=1$}
		\addlegendimage{mark = x}
		\addplot+[ mark = triangle,brown, mark indices = {13,30,50,80,135,200,265,321,351,371,388}] table[x index=0, y index=2, forget plot] {motivation_decay.dat};
		\addlegendentry{$n=2$}
		\addlegendimage{mark = triangle, brown}
		\addplot+[mark = diamond, blue, mark indices = {13,30,50,80,135,200,265,321,351,371,388}] table[x index=0, y index=3, forget plot] {motivation_decay.dat};
		\addlegendentry{$n=4$}
		\addlegendimage{mark = diamond, blue}
		\addplot+[mark=star, red, mark indices = {13,30,50,80,135,200,265,321,351,371,388}] table[x index=0, y index=4, forget plot] {motivation_decay.dat};
		\addlegendentry{$n=8$}
		\addlegendimage{mark = star, red}
		\end{axis}
		\end{tikzpicture}
		\captionsetup{width=0.99\linewidth}
		\vspace{-0.18cm}
		\captionof{figure}{\footnotesize Solution $u(x,2/3)$ for Dirichlet boundary conditions $-\cos(n \pi y)$ for $n=1,2,4,8$.}
		\label{figure_ex_decay}
	\end{minipage}
\end{figure}

The key ingredients to show compactness of the transfer operator are a parabolic Caccioppoli inequality and the compactness theorem of Aubin-Lions. The exponential decay of higher frequencies of the solutions from $I\times\partial\Omega^{out}$ to $I\times\Omega^{in}$ implies that the integral over the (spatial) gradient of a local solution on $I\times\Omega^{in}$ can be bounded in terms of the integral over the solution on $I\times\Omega^{out}$. This decay of energy in the interior of  $I\times\Omega^{out}$ is analytically captured by the Caccioppoli inequality that allows to bound the $L^2(I,H^1(\Omega^{in}))$-norm of local solutions in terms of their $L^2(I,L^2(\Omega^{out}))$-norm. Caccioppoli's inequality, closely linked to the exponential decay of the solutions in the interior, thus makes the local solution space amenable to approximation and facilitates the design of localizable multiscale methods. Moreover, the compactness theorem of Aubin-Lions is the parabolic analogon of the compact embedding of $H^1(\Omega^{out})$ in $L^2(\Omega^{out})$ in the stationary setting, and states that the space of functions in $L^2(I,H^1(\Omega^{out}))$ that have a time derivative in $L^2(I,H^{-1}(\Omega^{out}))$ is compactly embedded in $L^2(I,L^2(\Omega^{out}))$. In contrast to the elliptic setting \cite{BL11} the regularities do not match a priori and we therefore exploit an additional regularity result for the weak solutions to combine the two results.

Based on the approximation properties of the local ansatz space, we can furthermore show that the relative local approximation error is bounded by $\Vert P - \text{proj}_{\Lambda_k} P \Vert = \sigma_{k+1}$. Moreover, by employing a global coupling of the local ansatz spaces that allows to bound the global approximation error in terms of the local error contributions, we can achieve a global error that is decaying with the same rate. Therefore, our approach allows for local and global error control, and the local reduced spaces can be chosen such that a desired global error tolerance prescribed by the user is satisfied.

\subsection{Optimal local approximation spaces in the interior}
\label{interior}

Let $\Omega^{in}\subseteq\Omega^{out}$$\subseteq\Omega$ denote subdomains that are located in the interior of the computational domain satisfying $\dist(\partial\Omega^{out},\partial\Omega)>0$ and $\dist(\partial\Omega^{in},\partial\Omega^{out})>\delta>0$ as illustrated in \cref{inner_domains}. Since $\Omega^{in}$ and $\Omega^{out}$ lie in the interior of the spatial domain $\Omega$, we do not know the values of the global solution $u$ on $I\times\partial\Omega^{in}$ or $I\times\partial\Omega^{out}$. We only know that $u$ solves the linear heat equation locally in $I \times \Omega^{in}$ or $I \times\Omega^{out}$ with unknown Dirichlet boundary conditions on $I\times\partial\Omega^{in}$ or $I\times\partial\Omega^{out}$ as discussed in \cref{motivation_local}. As we do not want to make any assumptions about the geometry of the global domain $\Omega$ when constructing the local reduced models and want to choose the oversampling domain $\Omega^{out}$ as small as possible, we cannot make any assumptions on the values of $u$ on $I \times \partial \Omega^{out}$. Hence, we are interested in approximating all functions $ w\in L^\infty(I,L^2(\Omega^{in}))\cap L^2(I,H^1(\Omega^{in}))$ that satisfy $w(0)=u_0$ in $L^2(\Omega^{in})$ and solve
\begin{align}
	\begin{split}
		-\int_I (w(t),v)_{L^2(\Omega^{in})}\,&\varphi_t(t)\,dt + \int_I (\alpha\nabla w (t), \nabla v)_{L^2(\Omega^{in})}\,\varphi(t)\,dt \\
		&= \int_I \la f(t), v\ra_{H^1_0(\Omega^{in})}\,\varphi(t)\,dt \quad \forall \,v\in H^1_0(\Omega^{in}),\;\varphi \in C_0^\infty(I).
	\end{split}
	\tag{\text{$P^{in}$}}
\end{align}
The analogous problem on $I \times \Omega^{out}$ will be denoted by $(P^{out})$.

First, we address the case where $f=0$ and $u_0=0$ and discuss the general case at the end of this subsection. We consider the following spaces of functions:
\begin{eqnarray*}
	\HH^{in}&:=& \big\lb w\in L^\infty(I,L^2(\Omega^{in}))\cap L^2(I,H^1(\Omega^{in}))\, \big|\, w \text{ solves } (P^{in}) \text{ for } f=0,\;u_0=0 \big\rb,\nonumber\\
	\HH^{out}&:=& \big\lb w\in L^\infty(I,L^2(\Omega^{out}))\cap L^2(I,H^1(\Omega^{out}))\, \big|\, w \text{ solves } (P^{out}) \text{ for } f=0,\;u_0=0 \big\rb,\nonumber\\
	\BB^{out}&:=&\big\lb w\vert_{I\times\partial\Omega^{out}}\big| w \in \HH^{out}\big\rb = L^2(I,H^{1/2}(\partial\Omega^{out})).
\end{eqnarray*}

The trace theorem \cite[Theorem 2.1]{LM72} yields the existence of the traces in $\BB^{out}$.
We equip $\HH^{in}$ with the inner product $((u,v))_{in}:=\int_I \int_{\Omega^{in}} \alpha \nabla u \nabla v $ and the induced energy norm $\lenergy u \renergy_{in} := \Vert \alpha^\frac{1}{2}\nabla u \Vert_{L^2(I,L^2(\Omega^{in}))}$. Analogously, we equip $\HH^{out}$ with the energy norm $\Vert \alpha^\frac{1}{2}\nabla \cdot \Vert_{L^2(I,L^2(\Omega^{out}))}$. Furthermore, we equip $\BB^{out}$ with the inner product $((\mu,\nu))_{out}:=\int_I \int_{\Omega^{out}} \alpha \nabla H(\mu) \nabla H(\nu)$ and the induced energy norm $\lenergy \mu \renergy_{out} := \Vert \alpha^\frac{1}{2}\nabla H(\mu) \Vert_{L^2(I,L^2(\Omega^{out}))}$, where $H(\mu)\in \HH^{out}$ is the solution of $(P^{out})$ for $f=0$, $u_0=0$, and boundary condition $\mu\in \BB^{out}$.\footnote{Note that $\lenergy \cdot \renergy_{in}$ defines a norm on $\HH^{in}$ thanks to a Poincar\'e inequality for parabolic problems stated in \cref{par_poin_int}. Thanks to the trace inequality and \cref{par_poin_int} $\lenergy \cdot \renergy_{out}$ defines a norm on $\BB^{out}$.}

Since we are interested in approximating the space $\HH^{in}$, we next define a transfer operator $P:\BB^{out}\rightarrow \HH^{in}$, similar to \cite{BL11,SP16}, that is given by
\begin{align}
	P(w\vert_{I\times\partial\Omega^{out}}):=w\vert_{I \times \Omega^{in}}\quad\text{ for all } w \in \HH^{out} \text{ and thus }w\vert_{I\times\partial\Omega^{out}} \in \BB^{out}.
	\label{trans_op_int}
\end{align}

In order to approximate $\HH^{in}$ with the left singular vectors of $P$, we need to prove compactness of the latter; see \cref{compact_op_int}. To this end, we want to employ the compactness theorem of Aubin-Lions \cite[Corollary 5]{S86}, which states that the generalized Sobolev space $W^{1,2,2}(I,H^1(\Omega^{out}),H^{-1}(\Omega^{out})):= \lb u \in L^2(I,H^1(\Omega^{out})) \mid u_t \in L^2(I,H^{-1}(\Omega^{out})) \rb$ is compactly embedded in the space $L^2(I,L^2(\Omega^{out}))$, and the following parabolic Caccioppoli-type inequality\footnote{Similar parabolic Caccioppoli inequalities can, for instance, be found in \cite[(3.13)]{BD13}, \cite[Lemma 2.1]{C15}, \cite[Lemma 2.4]{K08}, and \cite[Lemma 2.4]{NP15}. }, which is proved in \cref{proofs_int}.

\begin{proposition}[Parabolic Caccioppoli inequality]\label{par_cacc_int}
	For a function $w\in \HH^{out}$ and thus $w\vert_{I \times \Omega^{in}}\in \HH^{in}$ the following estimate holds:
	\begin{align}
		\Vert w\vert_{I \times \Omega^{in}} \Vert^2_{L^\infty(I,L^2(\Omega^{in}))}+\Vert \alpha^\frac{1}{2}\nabla (w\vert_{I \times \Omega^{in}}) \Vert^2_{L^2(I\times\Omega^{in})}\leq \frac{8\alpha_1}{\delta^2}\Vert w \Vert^2_{L^2(I\times\Omega^{out})}. \label{caccioppoli_int}
	\end{align}
\end{proposition}

To combine \cref{par_cacc_int} with the compactness theorem of Aubin-Lions and thus prove compactness of $P$, the following regularity result is required, which is proved in \cref{proofs_int}.

\begin{lemma}[Regularity]\label{harmonic_prop_int}
	There holds $\HH^{in} \subseteq W^{1,2,2}(I,H^1(\Omega^{in}),H^{-1}(\Omega^{in}))$ and analogously for $\HH^{out}$.
\end{lemma}

As we aim at providing a good approximation space for a whole set of functions, the Kolmogorov $n$-width \cite{K36} serves as a benchmark and we will see below that the left singular vectors of $P$ actually span a space which is optimal in the sense of Kolmogorov; a notion which we define now.

\begin{definition}[Kolmogorov n-width]\label{kolmogorov}
	Let Hilbert spaces $(X,\Vert\cdot\Vert_X)$ and $(Y,\Vert\cdot\Vert_Y)$ with associated norms and a linear continuous operator $T:Y\rightarrow X$ be given. For an arbitrary $n\in\N$ let $X^n\subseteq X$ denote an n-dimensional subspace of $X$. Then, the Kolmogorov n-width of $T(Y)$ in $X$ is given by
	\begin{align*}
		d_n(T(Y);X):= \underset{dim(X^n)=n}{\underset{X^n\subseteq X}{\inf}}\underset{v\in Y}{\sup} \underset{w\in X^n}{\inf}  \frac{\Vert Tv-w\Vert_X}{\Vert v\Vert_Y}.
	\end{align*}
	An n-dimensional subspace $X^n \subseteq X$ is called optimal for $d_n(T(Y);X)$ if
	\begin{align*}
		\underset{v\in Y}{\sup}\, \underset{w\in X^n}{\inf} \, \frac{\Vert Tv-w\Vert_X}{\Vert v\Vert_Y} = d_n(T(Y);X).
	\end{align*}
\end{definition}

Having proved compactness of $P$, we finally introduce the corresponding adjoint operator $P^*:\HH^{in}\rightarrow\BB^{out}$. Consequently, their composition $P^*P:\BB^{out}\rightarrow\BB^{out}$ is a compact, self-adjoint, non-negative operator. Employing the Hilbert-Schmidt theorem as well as \cite[Theorem 2.2 in Chapter 4]{P85} then yields the following result:

\begin{theorem}[Optimal approximation spaces in the interior]\label{opt_spaces_int}
	Let $\lambda_i \in \R^+$ and $\varphi_i \in \HH^{out}$, $i=1,\ldots,\infty$, denote the eigenvalues and eigenfunctions satisfying the transfer eigenvalue problem: Find $(\lambda_i, \varphi_i) \in (\R^+ ,\HH^{out})$ such that
	\begin{align}\label{EW_Transfer_int}
		((\varphi_i\vert_{I \times \Omega^{in}}, w\vert_{I \times \Omega^{in}}))_{in} = \lambda_i\,((\varphi_i\vert_{I\times\partial\Omega^{out}},w\vert_{I\times\partial\Omega^{out}}))_{out}\quad\forall\,w\in\HH^{out}.
	\end{align}
	Furthermore, let the eigenvalues $\lb\lambda_i\rb_{i=1}^\infty $ be listed in non-increasing order of magnitude: $\lambda_1 \geq \lambda_2 \geq \ldots\geq 0$, additionally satisfying $\lambda_j  \xlongrightarrow{j \rightarrow \infty} 0$. Then the optimal approximation space for $d_n(P(\BB^{out});\HH^{in})$ is given by
	\begin{align*}
		\Lambda^n:= \spann \lb \chi_1,\ldots,\chi_n\rb,\quad \chi_i:=P(\varphi_i\vert_{I\times\partial\Omega^{out}}),\;i=1,\ldots,n.
	\end{align*}
	Moreover, the associated Kolmogorov n-width is characterized as follows:
	\begin{align}\label{kolm_weite_int}
		d_n(P(\BB^{out});\HH^{in})= \underset{v\in \BB^{out}}{\sup}\; \underset{w\in \Lambda^n}{\inf} \; \frac{\lenergy Pv-w\renergy_{in}}{\lenergy v\renergy_{out}}\, = \sqrt{\lambda_{n+1}}\,.
	\end{align}
\end{theorem}

\begin{proof}
	Thanks to the definition of the transfer operator $P$ and its adjoint operator $P^*$, the transfer eigenvalue problem \cref{EW_Transfer_int} may be reformulated as follows:
	\begin{align*}
		((P^*P\varphi_i\vert_{I\times\partial\Omega^{out}}, w\vert_{I\times\partial\Omega^{out}}))_{out}=\lambda_i\,((\varphi_i\vert_{I\times\partial\Omega^{out}},w\vert_{I\times\partial\Omega^{out}}))_{out}.
	\end{align*}
	Then, the assertion directly follows from Theorem 2.2 in Chapter 4 of \cite{P85}.
\end{proof}

Similar results have been obtained in \cite{BL11,SP16} for elliptic problems.\\
To address non-homogeneous data $f$ and $u_0$, $u^{f} \in L^\infty(I,L^2(\Omega^{out}))\cap L^2(I,H^1_0(\Omega^{out}))$ denotes the solution of $(P^{out})$ for arbitrary $f\in L^2(I,H^{-1}(\Omega^{out}))$ and $u_0\in L^2(\Omega^{out})$. Finally, the optimal local approximation space over $I \times \Omega^{in}$ is given by
\begin{align}
	\Lambda^{n,\text{data}}:=\spann \lb \chi_1,\ldots,\chi_n,\chi^{f} \rb, \quad \chi^{f}:=u^{f}\vert_{I \times \Omega^{in}}- \sum_{k=1}^n ((u^{f}\vert_{I \times \Omega^{in}},\chi_k))_{in} \,\chi_k.
	\label{optimal_space_data}
\end{align}
Thus, $ \Lambda^{n,\text{data}}$ provides an approximation of all functions in $\lb w\in L^\infty(I,L^2(\Omega^{in}))\cap L^2(I,H^1(\Omega^{in}))\, | \, w \text{ solves } (P^{in}) \rb$.


\subsection{Optimal local approximation spaces at the boundary}
\label{boundary}

Let $\Omega^{in}\subseteq \Omega^{out} \subseteq \Omega$ denote subdomains that are located at the boundary of the computational domain $\Omega$ and satisfy $\partial \Omega^{in} \cap \partial \Omega^{out} \cap \partial \Omega \neq \emptyset$ and $\dist(\partial\Omega^{in}\cap \Omega, \partial\Omega^{out}\cap \Omega)>\delta>0$ as illustrated in \cref{inner_domains_boundary} (right); the case $\partial \Omega^{out} \cap \partial \Omega \neq \emptyset$ and $\dist(\partial\Omega^{in}, \partial\Omega^{out})>\delta>0$ (see \cref{inner_domains_boundary} (left)) follows analogously.
\begin{figure}
	\centering 
	\includegraphics[height=1.5cm]{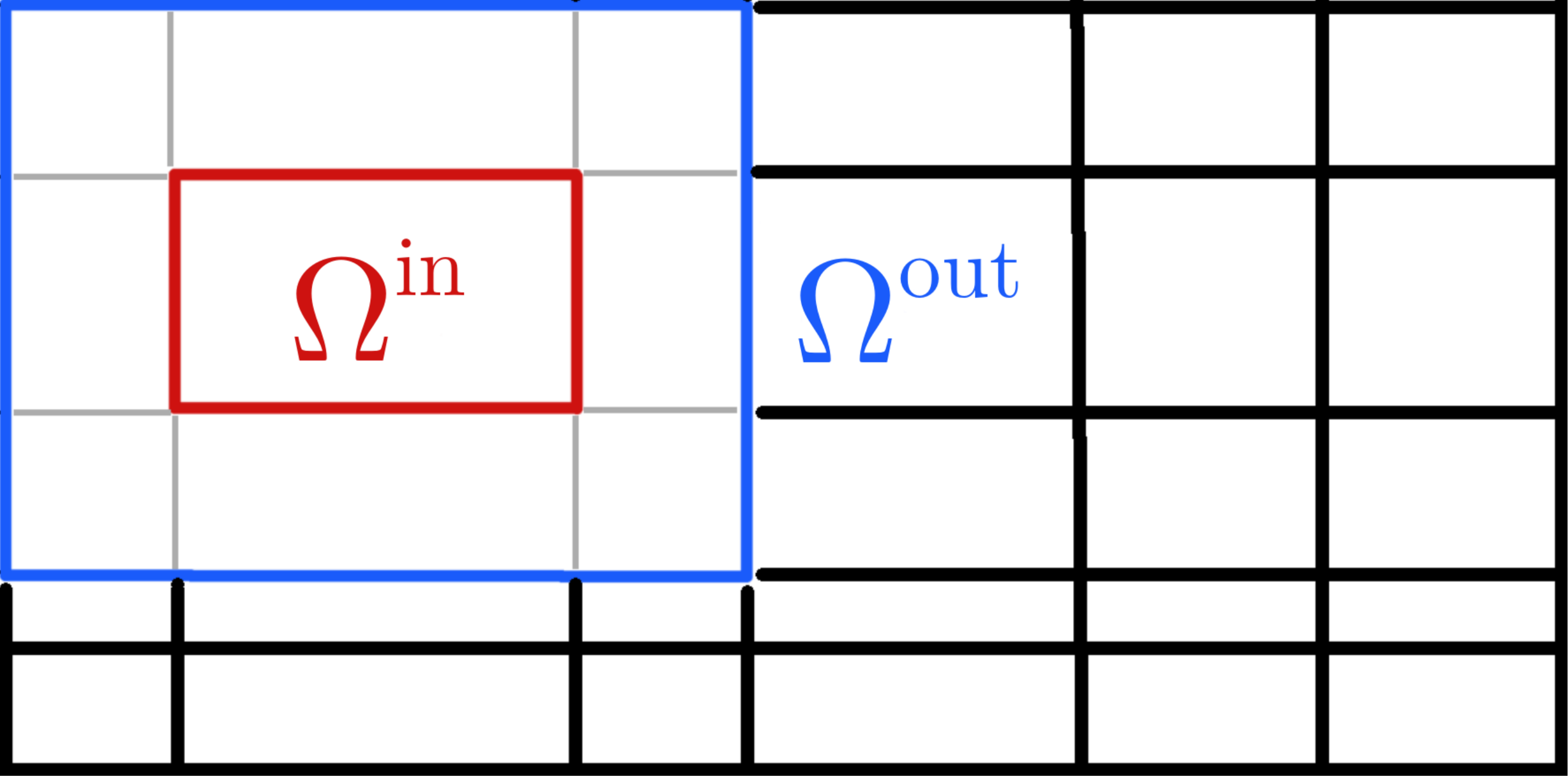}
	\hspace{1cm}
	\includegraphics[height=1.5cm]{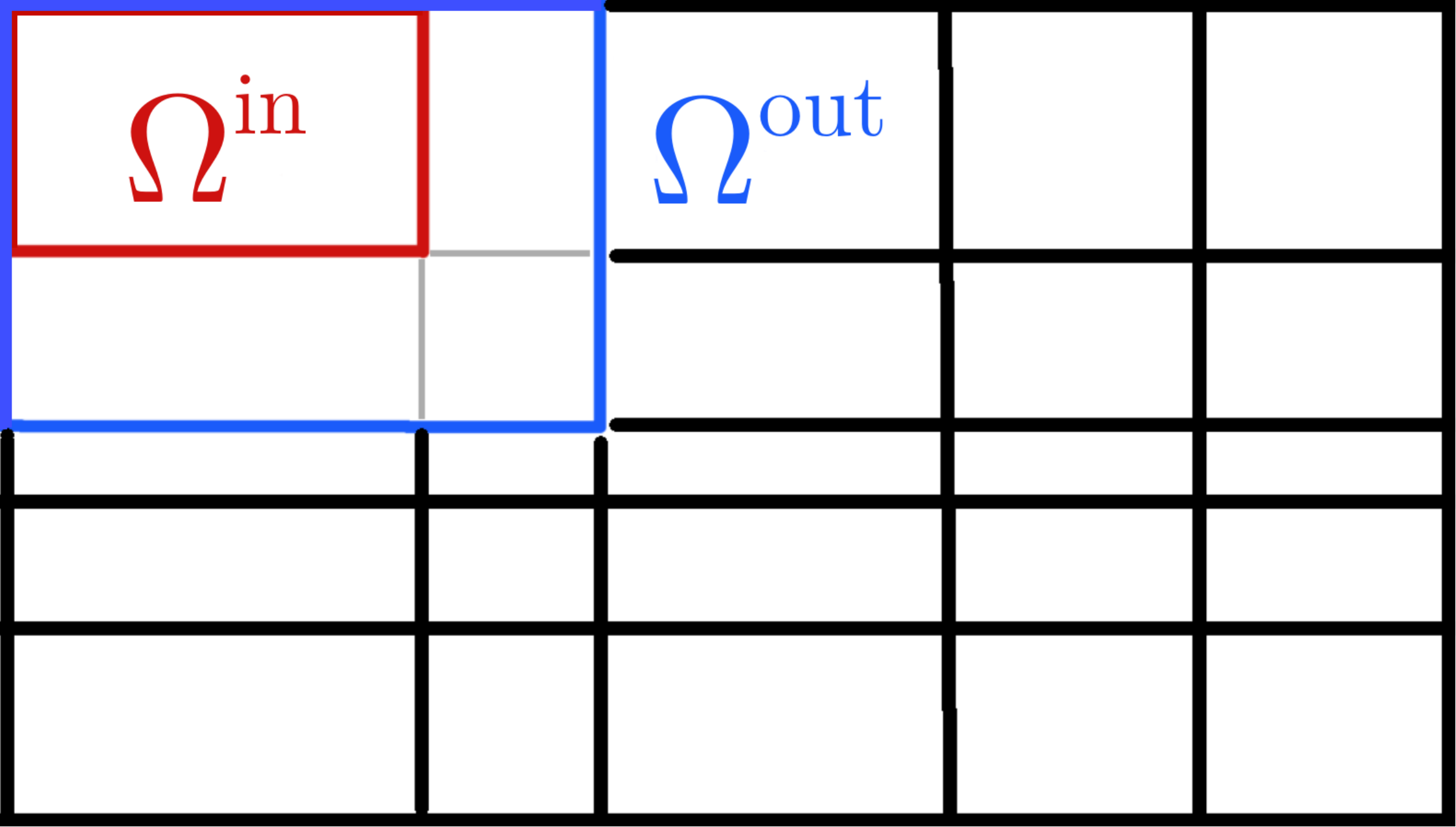}
	\captionsetup{width=0.95\linewidth}
	\captionof{figure}{\footnotesize Local target domain $\Omega^{in}$ and oversampling domain $\Omega^{out}$ at the boundary of $\Omega$.}	
	\label{inner_domains_boundary}
\end{figure}
We want to approximate all functions $w\in L^\infty(I,L^2(\Omega^{in}))\cap L^2(I,V_{\Sigma_D}^{in})$ that satisfy $w(0)=u_0$ in $L^2(\Omega^{in})$ and solve
\begin{equation}
\begin{split}
&-\int_I (w(t),v)_{L^2(\Omega^{in})}\,\varphi_t(t)\,dt + \int_I (\alpha \nabla w(t), \nabla v)_{L^2(\Omega^{in})}\,\varphi(t)\,dt \nonumber\\
= \int_I &\la f(t), v\ra_{V^{in}_0} \,\varphi(t)\,dt + \int_I \la g_N(t), v\ra_{H^\frac{1}{2}(\partial\Omega^{in}\cap\Sigma_N)}\,\varphi(t)\,dt  \quad \forall \,v\in V^{in}_0,\,\varphi \in C_0^\infty(I),\hspace{-0.6cm}
\end{split} \tag{\text{$P^{in}$}}
\end{equation}
where $V_{\Sigma_D}^{in}\mspace{-2mu}:= \mspace{-2mu}\lb w\mspace{-2mu} \in\mspace{-2mu} H^1(\Omega^{in}) \mid w \mspace{-2mu}=\mspace{-2mu} g_D \text{ on }\partial\Omega^{in}\mspace{-2mu}\cap\mspace{-2mu}\Sigma_D \rb$ and $V^{in}_0:= \lb w \mspace{-2mu}\in\mspace{-2mu} H^1(\Omega^{in}) \mid w\mspace{-2mu} =\mspace{-2mu} 0$ $ \text{on }\partial\Omega^{in} \mspace{-2mu}\cap\mspace{-2mu} (\Omega \mspace{-2mu}\cup \mspace{-2mu}\Sigma_D)\rb$. The analogous problem on $I\times\Omega^{out}$ will be denoted by $(P^{out})$.
We again first address the case where $f=0$, $u_0=0$, $g_D=0$, and $g_N=0$ and consider
\begin{align*}
&\HH^{in}:= \big\lb w\in L^\infty(I,L^2(\Omega^{in}))\cap L^2(I,V^{in}) \big| w \text{ solves } (P^{in}) \text{ for } f,u_0,g_D,g_N=0 \big\rb,\\
&\text{where } V^{in}:=\big\lb w \in H^1(\Omega^{in}) \big| w = 0 \text{ on } \partial\Omega^{in}\cap\Sigma_D \big\rb,\; \HH^{out} \text{ is defined analogously},\\
&\BB^{out} := \big\lb w\vert_{I\times\partial\Omega^{out}}\big| w \in \HH^{out}\big\rb.
\end{align*}
We introduce the transfer operator $P:\BB^{out} \rightarrow \HH^{in}$ defined as 
\begin{align}\label{trans_op}
P(w\vert_{I\times \partial\Omega^{out}}):=w\vert_{I\times\Omega^{in}}\quad \text{ for all } w\in \HH^{out},
\end{align}
where the compactness of $P$ follows from similar arguments as above (see \cref{proofs_bound}). Analogous to \cref{interior} we then obtain that the optimal approximation space for $d_n(P(\BB^{out});\HH^{in})$ is given by $\Lambda^n:= \spann \lb \chi_1,\ldots,\chi_n\rb $, $\chi_i:=P(\varphi_i\vert_{I\times \partial\Omega^{out}}) $,
where $\lambda_i\in \R^+$ and $\varphi_i\in \HH^{out}$ denote the $n$ largest eigenvalues and corresponding eigenfunctions that satisfy $P^*P\varphi_i = \lambda_i \varphi_i$. Moreover, it again follows that the Kolmogorov n-width is given by $d_n(P(\BB^{out});\HH^{in}) =  \sqrt{\lambda_{n+1}}$.

To address non-homogeneous data $f$, $u_0$, $g_D$, and $g_N$, first choose a lifting function $u^{b}\in W^{1,2,2}(I,V_{\Sigma_D},V^*)$\footnote{Recall that $V_{\Sigma_D}:= \lb w \in H^1(\Omega) \mid w = g_D \text{ on } \Sigma_D \rb$ and $V:=\lb w \in H^1(\Omega) \mid w = 0 \text{ on }\Sigma_D\rb$.} that satisfies $\alpha \nabla u^{b} \cdot n= g_N$ in $L^2(I,H^{-\frac{1}{2}}(\Sigma_N))$, $u^{b}(0)=0$ in $L^2(\Omega)$, and is equal to zero on the union of all local oversampling domains that do not touch the global boundary.\footnote{Note that in practice this is not a problem since we can for instance choose $u^b$ to be zero on all inner degrees of freedom and thus do not violate locality.} Next, for $f\in L^2(I,(V^{out}_0)^*)$ and $u_0\in L^2(\Omega^{out})$ let $u^{f} \in L^\infty(I,L^2(\Omega^{out}))\cap L^2(I,V^{out}_0)$ be the solution of $(P^{out})$, where the right hand side of $(P^{out})$ is given by 
\begin{align*}
\int_I \la f(t), v\ra_{V^{out}_0}\,\varphi(t)\,dt + \int_I (u^{b}(t),v)_{L^2(\Omega^{out})}\,\varphi_t(t)\,dt -  \int_I (\alpha \nabla u^{b} (t), \nabla v )_{L^2(\Omega^{out})}\,\varphi(t)\,dt
\end{align*}
for all  $v\in V^{out}_0$ and $\varphi \in C_0^\infty(I)$ and $V^{out}_0:= \lb w \in H^1(\Omega^{out}) \mid w = 0 \text{ on }\partial\Omega^{out} \cap (\Omega \cup \Sigma_D)\rb$. Then, the optimal approximation space over $I\times\Omega^{in}$ is denoted by $\Lambda^{n,\text{data}}:=\spann \lb \chi_1,\ldots,\chi_n,\chi^{f},u^{b}\vert_{I\times\Omega^{in}}\rb$, where $ \chi^{f}:=u^{f}\vert_{I\times\Omega^{in}}- \sum_{k=1}^n ((u^{f}\vert_{I\times\Omega^{in}},\chi_k))_{in} \,\chi_k$.


\section{Approximation of the optimal local approximation spaces}
\label{impl_and_random}
In this section we describe how to compute an approximation of the transfer eigenvalue problem (cf. \cref{EW_Transfer_int}) and thus an approximation of the optimal local reduced space $\Lambda^{n,\text{data}}$ (cf. \cref{optimal_space_data}) employing the FE method, and discuss its practical realization via Krylov subspace methods and random sampling.

\subsection{Computational realization of the transfer eigenvalue problem}
\label{implementation}
To simplify the notation we consider in this section subdomains $\Omega^{in}\subseteq\Omega^{out}$ located in the interior of $\Omega$ and homogeneous initial conditions. We assume that a partition of $\Omega^{out}$ is given such that $\partial\Omega^{in}$ does not intersect any element of that partition and consider a conforming finite element (FE) space $V_h^{out}\subseteq H^1(\Omega^{out})$ of dimension $N_h^{out}\in\N$. Furthermore, we introduce a triangulation of the time interval $I=(0,T)$ and consider a piecewise linear FE space $S_t\subseteq H^1(I)$ and a piecewise constant FE space $Q_t\subseteq L^2(I)$ of dimension $N_T\in \N$.\footnote{We assume that homogeneous initial conditions are included in the definition of $S_t$. For coefficients $\alpha$ that do not depend on the temporal variable, the Petrov-Galerkin formulation is equivalent to a Crank Nicolson time stepping procedure \cite{UP14}.} Then, the ansatz and test space of dimension $N:= N_T \cdot N_h^{out} $ for the Petrov-Galerkin approximation are given by $S_t \otimes V_h^{out}$ and $Q_t \otimes V_h^{out}$ with canonical basis functions $\phi_i$ and $\psi_i$ for $i=1,\ldots,N$. To ensure stability, we assume that the space-time discretization satisfies the CFL condition (cf. \cite{A12}).\footnote{Unconditional stability can be proven, for instance, if the discrete ansatz space is a subspace of the discrete test space or if the test space is refined with respect to the temporal direction. However, in the latter case the dimension of the test space will be larger than the dimension of the ansatz space and the discrete solution needs to be computed as the minimizer of a certain residual functional (see, e.g., \cite{A12}).} We introduce the matrix $\mathbf{B}\in \R^{N \times N}$ defined as 
\begin{align*}
\mathbf{B}_{ij}\mspace{-1mu}:=\mspace{-1mu} ((\phi_j)_t,\psi_i)_{L^2(I\times\Omega^{out})}\mspace{-1mu} +\mspace{-1mu} (\alpha \nabla \phi_j, \nabla \psi_i)_{L^2(I\times\Omega^{out})}\;\,\forall \phi_j \in S_t \mspace{-1mu}\otimes\mspace{-1mu} V_h^{out}, \psi_i \in Q_t \mspace{-1mu}\otimes \mspace{-1mu}V_h^{out}
\end{align*}
for $i,j=1,\ldots,N$. In addition, we suppose that the $N_{out}$ rows of $\mathbf{B}$ that correspond to nodes on $I\times\partial\Omega^{out}$ are replaced by the respective rows of the identity matrix. For every $m=1,\ldots,N_{out}$ we then compute the local solution $\mathbf{u}_m\in\R^{N}$ by solving
\begin{align}\label{local_comp_1}
\mathbf{B}\, \mathbf{u}_m = \mathbf{e}_m,
\end{align}
where $\mathbf{e}_m\in \R^{N}$ is the unit vector associated with the m-th node that lies on $I\times\partial\Omega^{out}$. Moreover, we denote by  $\mathbf{M}_{in}\in \R^{N_{in} \times N_{in}}$ and $\mathbf{M}_{out} \in \R^{N_{out}\times N_{out}} $ inner product matrices associated with $I\times\Omega^{in}$ and $I\times\partial\Omega^{out}$ and $\mathbf{D}_{\,\rightarrow \text{in}} \in \R^{N_{in}\times N}$ restricts the coefficients of a FE function on $I\times\Omega^{out}$ to the respective $N_{in}$ coefficients corresponding to $I\times\Omega^{in}$. Finally, we assemble and solve the following generalized eigenvalue problem: Find eigenvalues $\lambda_{h,j} \in \R^+$ with corresponding eigenvectors $\boldsymbol{\xi}_j\in \R^{N_{out}}$ such that\footnote{Assuming a certain ordering of the nodes, the matrix containing the restrictions of $\mathbf{u}_{1},\ldots,\mathbf{u}_{N_{out}}$ to $I\times\partial\Omega^{out}$ on the right hand side simplifies to the identity matrix.}
\begin{equation}\label{local_comp_2}
\left[\mathbf{D}_{\,\rightarrow \text{in}}\,  \mathbf{u}_{1}  \ldots  \mathbf{D}_{\,\rightarrow \text{in}} \,\mathbf{u}_{N_{out}} \right]^\top \mathbf{M}_{in} \left[ \mathbf{D}_{\,\rightarrow \text{in}}\, \mathbf{u}_{1}\ldots\mathbf{D}_{\,\rightarrow \text{in}}\, \mathbf{u}_{N_{out}} \right] \boldsymbol{\xi}_j
= \lambda_{h,j}\, \mathbf{M}_{out}\, \boldsymbol{\xi}_j.
\end{equation}
The coefficients of the FE approximation of the reduced basis $\lbrace \chi_1,\ldots,\chi_n\rbrace $ are then given by $\boldsymbol{\chi}_j :=   \left[ \mathbf{D}_{\,\rightarrow \text{in}}\,\mathbf{u}_{1}\;\ldots\; \mathbf{D}_{\,\rightarrow \text{in}}\,\mathbf{u}_{N_{out}} \right] \boldsymbol{\xi}_j \in \R^{N_{in}}$ for $j=1,\ldots,n$. 

Moreover, the transfer eigenvalue problem can equivalently be approximated using the matrix representation $\mathbf{P}\mspace{-2mu}\in\mspace{-2mu}\R^{N_{out} \times N_{in}}$ of the discrete transfer operator $P_h$ given as
\begin{align}\label{discrete_trans_op}
\mathbf{P} \,\boldsymbol{\xi} = \mathbf{D}_{\,\rightarrow \text{in}} \, \mathbf{B}^{-1}\, \mathbf{D}_{\,\text{out}\rightarrow} \,\boldsymbol{\xi}.
\end{align}
Here, $ \mathbf{D}_{\,\text{out}\rightarrow} \in \R^{N\times N_{out}}$ extends a coefficient vector corresponding to $I\times\partial\Omega^{out}$ to the respective coefficient vector associated with $I\times\Omega^{out}$ by setting the new entries to zero. Consequently, we solve the problem: Find $\lambda_{h,j} \in \R^+$ and $\boldsymbol{\xi}_j\in \R^{N_{out}}$ such that
\begin{align}\label{discrete_eigenproblem}
\mathbf{P}^\top \,\mathbf{M}_{in} \,\mathbf{P} \,\boldsymbol{\xi}_j = \lambda_{h,j} \,\mathbf{M}_{out}\,\boldsymbol{\xi}_j.
\end{align}
The reduced basis FE coefficients are then given by $\boldsymbol{\chi}_j := \mathbf{P}\, \boldsymbol{\xi}_j \in \R^{N_{in}}$ for $j=1,\ldots,n$.

To augment the reduced FE basis with the data $f$ we compute the solution of $\mathbf{B}\, \mathbf{u}^{f} = \mathbf{F}$, where $\mathbf{F}_i := 0$ for indices $i$ that correspond to nodes on $I\times\partial\Omega^{out}$ and $\mathbf{F}_i := (f,v_i)_{L^2(I\times\Omega^{out})}$, $v_i \in Q_t\otimes V_h$, else. Then, $\mathbf{D}_{\,\rightarrow \text{in}}\,\mathbf{u}^{f}$ is added to the basis.

\begin{remark}[Comparison of computational approaches to approximate the optimal local spaces] \label{remark_complexity}
	A direct computation of the optimal local space via \cref{local_comp_1} and \cref{local_comp_2} would require $N_{out}$ evaluations of the transfer operator and thus $N_{out}$ local solutions of the PDE and solving a dense generalized eigenproblem of dimension $N_{out}\times N_{out}$. As this becomes infeasible for large $N_{out}$, one would in general use Krylov subspace or randomized methods for the approximation of the transfer eigenvalue problem. 
	
	In Krylov subspace methods, the application of the transfer operator \cref{discrete_trans_op} would be implicitly passed to the eigenvalue solver. To calculate the $n$ eigenvectors corresponding to the biggest $n$ eigenvalues of  $P_h^\top P_h$ using, for instance, the implicitly restarted Arnoldi method (IRAM) from \cite{LS98}, $\mathcal{O}(n)$ evaluations of $P_h$ and $P_h^\top$ are required in every iteration. Algorithm 1 in subsection SM2.1 summarizes the approximation of the optimal local spaces using IRAM.
	While Krylov subspace methods can lead to more accurate approximations especially for slowly decaying singular values, randomized methods have the main advantage that they are inherently stable and amenable to parallelization. 
	
	To generate an approximation space of dimension $n$ via random sampling as described in subsections SM2.2 and \ref{randomization}, $n + n_t$ evaluations of the transfer operator are required in total; $n$ evaluations to construct the basis and $n_t$ evaluations to construct test vectors that are used for a probabilistic a posteriori error estimator. In comparison, the $\mathcal{O}(n)$ evaluations of $P_h$ and $P_h^\top$ in every iteration of IRAM will likely sum up to more than $n+n_t$ evaluations required by random sampling.
	
	As randomized methods can outperform Krylov subspace methods even in the sequential setting (see, e.g., \cite{BS18}), they are thus an appealing choice for the approximation of the optimal local spaces. For a more in-depth comparison of Krylov subspace and randomized methods, we refer, for instance, to \cite[section 6]{HM11}.
\end{remark}

\begin{remark}[Computational complexity] \label{remark_comp_complexity}
		The computational complexity of the local basis construction is clearly dominated by the evaluation of $P_h$ or $P_h^\top$ and thus the numerical solution of the local PDE, where we employ a sparse direct solver. For a standard space-time FE discretization in three dimensions (one temporal and two spatial dimensions) the factorization of $\mathbf{B}\in\R^{N\times N}$ (e.g., LU, QR, Cholesky) can be computed in $\mathcal{O}(N^2)$ work \cite{GV13}. After factorizing, the computational complexity for each local solution of the PDE is $\mathcal{O}(N^{4/3})$ \cite{GV13}. More details on the computational complexity of the local basis construction are provided in subsection SM2.3.
\end{remark}

\subsection{Quasi-optimal local approximation spaces via random sampling}
\label{randomization}

To construct a suitable approximation $\Lambda^n_\text{rand}$ of the optimal local space $\Lambda^n$ introduced in \cref{local_appr_spaces}, we prescribe random boundary conditions on $I\times\partial\Omega^{out}$, where the coefficient vectors of the corresponding FE functions are normally distributed with mean zero and covariance matrix $\mathbf{M}_{out}^{-1}$. Then, $\mathbf{P}$ is applied to the random vectors and $\Lambda^n_\text{rand}$ is spanned by the resulting local solutions. Using this approach we can achieve a quasi-optimal convergence of order $\sqrt{n\,\lambda_{h,n+1}}$ \cite{BS18,HM11}. Furthermore, we employ an adaptive algorithm that is driven by a probabilistic a posteriori error estimator. The output of the algorithm is an approximation space $\Lambda^n_\text{rand}$ that satisfies the property $\mathbb{P}(\Vert P_h-\text{proj}_{\Lambda^n_\text{rand}} P_h\Vert \leq \tol) > (1 - \epsalgo)$,\footnote{Here, $\text{proj}_{\Lambda^n_\text{rand}}$ denotes the orthogonal projection onto the space $\Lambda^n_\text{rand}$.} where the accuracy $\tol$ and the failure probability $\epsalgo$ are prescribed by the user. For further details we refer to subsection SM2.2 and to \cite{BS18} where methods from randomized linear algebra \cite{HM11} have been used to approximate the optimal local approximation spaces in the elliptic setting.


\section{Global approximation}
\label{global_approximation}

To compute an approximation of the global solution, we employ the GFEM as one example for coupling the local approximation spaces introduced in \cref{local_appr_spaces} since it allows to bound the global approximation error in terms of the local error contributions. Exploiting the local optimality result \cref{kolm_weite_int}, which states that $\Vert P - \text{proj}_{\Lambda_n} P \Vert = \sigma_{n+1}$\footnote{Recall that $\sigma_{n+1}$ denotes the $n+1$st singular value of $P$ and $\text{proj}_{\Lambda_n}$ is the orthogonal projection onto the space $\Lambda_n$ spanned by the $n$ leading left singular vectors of $P$.}, we first show in  \cref{error_bounds} that the relative local approximation error in the $L^2(H^1)$-seminorm is bounded by the projection error (which equals $\sigma_{n+1}$) times a locally computable constant (cf. \cref{a_priori}). Subsequently, we prove in \cref{global_error} that the global error between the solution and the GFEM approximation in a suitable graph norm can be bounded only by the local errors in the $L^2(H^1)$-seminorm and is thus decaying as $\sigma_{n+1}$ or better. If we employ random sampling to approximate the optimal local spaces, the local and global a priori error bounds are still valid, provided that $\Vert P_h-\text{proj}_{\Lambda^n_\text{rand}} P_h\Vert \leq \varepsilon$ for a local error tolerance $\varepsilon>0$ with a very low failure probability (cf. \cref{randomization}), and we can achieve a local and global error decaying as $\sqrt{n}\, \sigma_{h,n+1}$ or better (cf. Theorem SM2). A priori error bounds concerning the elliptic case can be found in \cite{BH14,BL11,BS18,SP16}. Finally, we discuss the computational realization of the global GFEM approximation in \cref{implementation_global}.

In contrast to existing approaches for parabolic PDEs \cite{GS00}, we consider here a space-time GFEM based on local space-time ansatz functions. Consequently, the computation of the global approximation does not require a time stepping method and is thus computationally very efficient due to the very rapid and exponential decay of the local and global approximation errors observed in several numerical experiments (cf. \cref{numerical_experiments}). In contrast, for certain problems, for instance, problems with multiscale diffusion coefficients that are varying non-periodically in time, a reduction only with respect to the spatial variable can become expensive if the time discretization involves many time steps. As either the reduced spatial bases would become very large since snapshots for many time points have to be included or one would have to use an adaptive-in-time procedure based on smaller time intervals, a very large amount of memory would be consumed.

To this end, let $\lbrace\Omega^{in}_1,\ldots,\Omega^{in}_M\rbrace$ be an open cover of the computational domain $\Omega$ such that $\Omega = \cup_{i=1}^M \Omega^{in}_i$. We introduce a corresponding partition of unity $\lbrace \psi_1,\ldots,\psi_M\rbrace$ with the following properties for $i=1,\ldots,M$ and $0< c_1,c_2 < \infty$: 
\begin{align} \label{properties_pou}
\begin{split}
&\psi_i \in C^1(\Omega^{in}_i), \quad
\supp(\psi_i)\subseteq \overline{\Omega^{in}_i},\quad
\sum_{i=1}^M \psi_i(x) = 1 \quad\forall\,x\in \Omega,\\
&\Vert \psi_i \Vert_{L^\infty(\Omega)} \leq c_1,\quad
\Vert \nabla \psi_i \Vert_{L^\infty(\Omega)} \leq c_2 / \diam(\Omega^{in}_i).
\end{split}
\end{align}
Then, we define the global GFEM ansatz space as $X_\mathrm{GFEM}:=\oplus_{i=1}^M \lbrace \psi_i u_i \mid u_i \in \Lambda_i^{n,\text{data}}\rbrace$, where $\Lambda_i^{n,\text{data}}$ is the local reduced space on $I\times\Omega_i^{in}$ (cf. \cref{local_appr_spaces}). The oversampling domains used to construct the local reduced spaces will be denoted by $I\times\Omega^{out}_i$ for $i=1,\ldots,M$. Moreover, we assume that we have the following overlap conditions: There exist $M^{in},M^{out}\in\N$ such that for every $x\in\Omega$ we have that 
\begin{align}\label{overlap_cond}
\#\lbrace i \mid x \in \Omega^{in}_i\rbrace\leq M^{in} \quad\text{and}\quad \#\lbrace i \mid x \in \Omega^{out}_i\rbrace\leq M^{out}.
\end{align}
Finally, let a test space $V_\mathrm{GFEM} \subseteq L^2(I,V)$\footnote{Recall that $V:=\lb w \in H^1(\Omega) \mid w = 0 \text{ on }\Sigma_D\rb $ with $\Vert \cdot \Vert_V := \Vert \alpha^\frac{1}{2} \cdot \Vert_{L^2(\Omega)}+\Vert \alpha^\frac{1}{2}\nabla \cdot \Vert_{L^2(\Omega)}$.} be given such that the inf-sup condition 
\begin{align}\label{inf_sup}
\beta = \inf_{u \in X_\mathrm{GFEM}} &\sup_{v \in V_\mathrm{GFEM}} \frac{\la u_t, v \ra_{L^2(I,V)} + (\alpha \nabla u, \nabla v)_{L^2(I,L^2(\Omega))} }{\Vert \alpha^\frac{1}{2}\nabla u \Vert_{L^2(I,L^2(\Omega))} \Vert v \Vert_{L^2(I,V)}} > \beta_0>0
\end{align}
is satisfied. Then, the global Petrov-Galerkin GFEM approximation reads as follows: Find $u_\mathrm{GFEM}\in X_\mathrm{GFEM}$ such that
\begin{align}\label{GFEM_formulation}
\begin{split}
\la (u_\mathrm{GFEM})_t, v \ra_{L^2(I,V)} &+ (\alpha \nabla u_\mathrm{GFEM}, \nabla v)_{L^2(I,L^2(\Omega))} \\
&=  \la f, v \ra_{L^2(I,V)} + \la g_N(t), v\ra_{L^2(I,H^\frac{1}{2}(\Sigma_N))} \quad \forall \, v \in V_\mathrm{GFEM}. 
\end{split} 
\end{align}

\begin{remark}
		In general, we cannot guarantee stability of the discretization. A well-known strategy to ensure stability is to construct optimally stable pairs of ansatz and test spaces (see, e.g., \cite{And13, BS19,DH12, DemGop10, DemGop11,DieSto21,SteWes21}). This, however, may require global computations and a localization strategy is thus the subject of future work. Nevertheless, \cref{table_inf_sup} shows that for our particular choice of test and trial spaces (c.f. \cref{implementation_global}) the inf-sup constant $\beta$ is close to one and we thus have stability at least for the considered numerical test cases (cf. \cref{global_numerics}).
\end{remark}

\subsection{A priori error bounds}
\label{error_bounds}
In the remainder of the paper all local quantities that are associated with the subdomains $I\times\Omega_i^{in}$ and $I\times\Omega_i^{out}$ will be identified by the additional subscript $i$. Moreover, we assume that the data functions satisfy $f \in L^2(I,L^2(\Omega))$, $ g_D \in L^2(I,H^{3/2}(\Sigma_D))\cap H^{3/4}(I,L^2(\Sigma_D))$ with $g_D\vert_{t=0}=0$, and $g_N \in L^2(I,H^{1/2}(\Sigma_N))\cap H^{1/4}(I,L^2(\Sigma_N))$. The assumptions on $g_D$ and $g_N$ guarantee that the lifting function $u^b$ defined in \cref{boundary} has a time derivative in $L^2(I,L^2(\Omega))$ \cite[Theorem 2.1]{LM72}.\footnote{If $\Sigma_N = \emptyset$ it is sufficient to assume that $ g_D \in L^2(I,H^{1/2}(\partial\Omega))\cap H^{1/2}(I,L^2(\partial\Omega))$ (cf. \cite[Theorem 2.1]{LM72}). We therefore conjecture that also for $g_N$ less regularity is sufficient to infer that $u^b_t \in L^2(I,L^2(\Omega))$, but unfortunately we could not find a result guaranteeing this property.} In the proof of the global a priori error bound we exploit the local a priori error result and the previous assumptions on the data allow us to estimate the sum of local norms of certain data terms by their respective global norms.

\begin{proposition}[Local a priori error bound] \label{a_priori}
	Let $\Omega_i^{in}\subseteq\Omega_i^{out}\subseteq\Omega$ be subdomains as introduced in \cref{interior,boundary} and let $u$ be the (global) solution of \cref{global_formulation}. If the reduced space $\Lambda_i^n$ over $I\times\Omega^{in}_i$ for $\varepsilon_i>0$ satisfies 
	\begin{align}\label{proj_error_epsilon}
	\Vert P_i  - \proj_{\Lambda_i^n} P_i \Vert = \underset{v\in\BB^{out}_i}{\sup}\,\underset{w \in \Lambda_i^n}{\inf} \frac{\Vert \alpha^\frac{1}{2}\nabla( P_i v - w )\Vert_{L^2(I,L^2(\Omega_i^{in}))}}{\Vert \alpha^\frac{1}{2} \nabla H_i (v) \Vert_{L^2(I,L^2(\Omega_i^{out}))}} \leq \varepsilon_i,
	\end{align}
	then there exists a $w_i^n \in \Lambda_i^{n,\text{data}}$ such that the following a priori error bound holds:
	\begin{align}
	\frac{\Vert \alpha^\frac{1}{2} \nabla (u\vert_{I\times\Omega_i^{in}} - w_i^n) \Vert_{L^2(I\times\Omega_i^{in})}}{\Vert \alpha^\frac{1}{2} \nabla u\vert_{I\times\Omega_i^{out}} \Vert_{L^2(I\times\Omega_i^{out})}\mspace{-3mu} +\mspace{-3mu} \Vert f \Vert_{L^2(I\times\Omega_i^{out})} \mspace{-3mu}+ \mspace{-3mu}\Vert u_0 \Vert_{L^2(\Omega_i^{out})}\mspace{-3mu} + \mspace{-3mu}\star_i}
	\leq  \max\lb 2, c_{f,i}\rb \, \varepsilon_i,
	\end{align}
	where $\star_i$ is given by $\star_i := \Vert u^b_t \Vert_{L^2(I,L^2(\Omega_i^{out}))} +\Vert \alpha^\frac{1}{2} \nabla u^{b} \Vert_{L^2(I,L^2(\Omega_i^{out}))}$ if $\Omega_i^{out}$ is located at the global boundary and $\star_i := 0$ else. Furthermore, the constant $c_{f,i}$ is defined as $c_{f,i}:=\Vert u^{f}_i \Vert_{L^2(I,L^2(\Omega_i^{out}))}/\Vert \alpha^\frac{1}{2} \nabla u^{f}_i \Vert_{L^2(I,L^2(\Omega_i^{out}))}$ and $ \proj_{\Lambda_i^n}$ is the orthogonal projection onto the space $\Lambda_i^n$.
\end{proposition}

\begin{proof}
	The key idea is to employ that functions in $\Lambda_i^{n,\text{data}}$ solve the PDE locally and to use Assumption \cref{proj_error_epsilon}. For a detailed proof see \cref{proofs_error_bounds}.
\end{proof}

\begin{remark}
	Note that Assumption \cref{proj_error_epsilon} holds either with $\varepsilon_i = \sqrt{\lambda_{n+1,i}} $ if we employ the optimal local approximation space (cf. \cref{kolm_weite_int}) or with a very low failure probability if we approximate the optimal local space via random sampling using Algorithm 2 (cf. subsections SM2.2 and \ref{randomization}). Since the transfer eigenvalue problems just contain spatial derivatives, only the approximation error in the $L^2(H^1)$-seminorm can locally be bounded via \cref{opt_spaces_int}. Moreover, one can explicitly construct counterexamples demonstrating that the local a priori error result does in general \emph{not} hold without additional data terms or additional assumptions on the data in the parabolic case as the following remark shows.
\end{remark}

\begin{remark} \label{counter_example}
	In the elliptic case (cf. \cite[Lemma SM5.2]{BS18}) the a priori error bound 
	\begin{align*}
	\Vert \alpha^\frac{1}{2} \nabla (u\vert_{\Omega^{in}} - w^n) \Vert_{L^2(\Omega^{in})}\leq c\,\Vert \alpha^\frac{1}{2} \nabla u\vert_{\Omega^{out}} \Vert_{L^2(\Omega^{out})}
	\end{align*}
	holds for a constant $c>0$. Here, a bound $\Vert \alpha^\frac{1}{2} \nabla (u\vert_{I\times\Omega^{in}} - w^n) \Vert_{L^2(I,L^2(\Omega^{in}))}\leq c\,\Vert \alpha^\frac{1}{2} \nabla u\vert_{I\times\Omega^{out}} \Vert_{L^2(I,L^2(\Omega^{out}))}$ does in general \emph{not} hold as the following example shows: Assume that $\Omega^{out}$ is located in the interior of $\Omega$ and consider the linear heat equation
	\begin{align*}
	\partial_t u - \Delta u = 1\;\; \text{in }I\times\Omega^{out},\;\;
	u(0,x)=0\;\; \forall\,x\in\Omega^{out},\;\;
	u(t,x)=t\;\; \forall\,(t,x)\in I\times\partial\Omega^{out},
	\end{align*}
	where $I\times\Omega^{out} =(0,1)\times (0,\pi)^2$. Then, $u$ equals $\tilde{u} + \hat{u}$, where $\tilde{u}$ and $\hat{u}$ solve
	\begin{align*}
	\partial_t \tilde{u} - \Delta \tilde{u} = 0\;\; \text{in } I\times\Omega^{out},&\;\;
	\tilde{u}(0,x)=0 \;\; \forall\, x\in\Omega^{out}, \;\;
	\tilde{u}(t,x)=t \;\; \forall\,(t,x)\in I\times\partial\Omega^{out}, \\
	\partial_t \hat{u} - \Delta \hat{u} = 1\;\; \text{in } I\times\Omega^{out},&\;\;
	\hat{u}(0,x)=0\;\; \forall \,x\in\Omega^{out},\;\;
	\hat{u}(t,x)=0\;\; \forall\,(t,x)\in I\times\partial\Omega^{out}.
	\end{align*}
	Using separation of variables we obtain (see section SM1)
	\begin{align*}
	\tilde{u}(t,x_1,x_2)&= \sum_{k,l=1}^\infty - \frac{4\,(1-\cos(k\pi)) (1-\cos(l\pi)) (1-e^{-(k^2+l^2)t})}{k\,l\,(k^2+l^2)\,\pi^2} \sin(k x_1) \sin(l x_2) + t,\\
	\hat{u}(t,x_1,x_2)&= -\tilde{u}(t,x_1,x_2) + t.
	\end{align*}
	Hence, $\nabla \hat{u}$ equals $- \nabla \tilde{u}$ and we have $\Vert \nabla u \Vert_{L^2(I\times\Omega^{out})}=0$ while $\Vert \nabla \tilde{u}\Vert_{L^2(I\times\Omega^{out})}>0$. Therefore, the estimate $\Vert \nabla \tilde{u}\Vert_{L^2(I\times\Omega^{out})}\leq \Vert \nabla u \Vert_{L^2(I\times\Omega^{out})}$ does not hold, but since $\tilde{u}$ equals $u - \hat{u}$ we can infer that $\Vert \nabla \tilde{u}\Vert_{L^2(I\times\Omega^{out})}\leq \Vert \nabla u \Vert_{L^2(I\times\Omega^{out})}+\Vert \nabla \hat{u} \Vert_{L^2(I\times\Omega^{out})}$ and the data corrector $\hat{u}$ can be estimated in terms of the data (cf. proof of \cref{a_priori}). However, in specific cases the local a priori error bound may be improved.
\end{remark}

\begin{proposition}[Global GFEM error bound] \label{global_error}
	Let $u$ be the solution satisfying \cref{global_formulation} and let $u_\mathrm{GFEM}\in X_\mathrm{GFEM}$ be the GFEM approximation solving \cref{GFEM_formulation}. If we assume that for each subdomain $\Omega_i^{in}$ and $\varepsilon_i>0$ there exists a $w_i^n \in \Lambda_i^{n,\text{data}}$ such that
	\begin{align}\label{a_priori_epsilon}
	\frac{\Vert \alpha^\frac{1}{2} \nabla (u\vert_{I\times\Omega_i^{in}} - w_i^n) \Vert_{L^2(I\times\Omega_i^{in})}}{\Vert \alpha^\frac{1}{2} \nabla u\vert_{I\times\Omega_i^{out}} \Vert_{L^2(I\times\Omega_i^{out})}\mspace{-3mu}  + \mspace{-3mu} \Vert f \Vert_{L^2(I\times\Omega_i^{out})}\mspace{-3mu} + \mspace{-3mu} \Vert u_0 \Vert_{L^2(\Omega_i^{out})} \mspace{-3mu} +\mspace{-3mu}  \star_i
	}
	\leq \max\lb 2, c_{f,i}\rb\, \varepsilon_i,
	\end{align}
	where $\star_i$ and $c_{f,i}$ are defined as in \cref{a_priori}, the following error bound holds:
	\begin{align*}
	&\sqrt{\Vert (u - u_\mathrm{GFEM})_t \Vert_{L^2(I,(V_\mathrm{GFEM})^*)}^2 + \Vert \alpha^\frac{1}{2} \nabla (u - u_\mathrm{GFEM}) \Vert_{L^2(I,L^2(\Omega))}^2 }\\ 
	\leq \;& \sqrt{10\, M^{out}} \, \max_{i=1,\ldots,M} \big\lbrace C_i(c_1,c_2,\beta,M^{in},\Omega_i^{in},c_{p,i}^{\alpha})\, \max\lb 2, c_{f,i}\rb\,\varepsilon_i \big \rbrace\\
	&\Big(\Vert \alpha^\frac{1}{2}\nabla u \Vert_{L^2(I\times\Omega)} +  \Vert f \Vert_{L^2(I\times\Omega)} + \Vert u_0 \Vert_{L^2(\Omega)} + \Vert u^b_t \Vert_{L^2(I\times \Omega)} +\Vert \alpha^\frac{1}{2} \nabla u^{b} \Vert_{L^2(I\times \Omega)}\Big).
	\end{align*}
	Here, the constant $C_i(c_1,c_2,\beta,M^{in},\Omega_i^{in},c_{p,i}^{\alpha})$ is given by $C_i(c_1,c_2,\beta,M^{in},\Omega_i^{in},c_{p,i}^{\alpha}):= \max \big\lb (1+1/\beta) \sqrt{2 M^{in} \big(c_1^2+(c_2 c_{p,i}^{\alpha} / \diam(\Omega^{in}_i))^2\big)},$ $\big(c_1 + c_2 / \diam(\Omega_i^{in})\big) /\beta \big\rb$ and the constant $c_{p,i}^{\alpha}$ is defined as $c_{p,i}^\alpha := \sup_{w_i \in \HH_i^{in}} \Vert \alpha^\frac{1}{2}  w_i \Vert_{L^2(I\times\Omega^{in}_i)} / \Vert \alpha^\frac{1}{2} \nabla w_i \Vert_{L^2(I\times\Omega^{in}_i)} $.
\end{proposition}

\begin{proof}
	The key idea is to exploit the local a priori error bound \cref{a_priori_epsilon} and the fact that the $\mathrm{GFEM}$ basis functions solve the PDE locally. To that end, the proof follows similar ideas as the proofs of Proposition SM5.1 and Corollary SM5.3 in \cite{BS18} and Theorem 2.1 in \cite{BM96}. To additionally control the time derivative of the global error in the $L^2((V_\mathrm{GFEM})^*) $-norm, we use a (Petrov-)Galerkin orthogonality of the approximation error and the reduced test space. A detailed proof is provided in \cref{proofs_error_bounds}.
\end{proof}

\begin{remark}
	\cref{global_error} shows that the global convergence of the GFEM approximation to the true solution with respect to a certain graph norm can be controlled only by the respective local errors in the $L^2(H^1)$-seminorm for both the randomized setting and the transfer eigenvalue problem (choosing $\varepsilon_i = \sqrt{\lambda_{n+1,i}} $). Moreover, it enables a localized construction of the local ansatz spaces such that the global GFEM approximation satisfies a prescribed global error tolerance (with a very low failure probability in the randomized setting), cf. Algorithm 3 in section SM3. We only have one global constant, the reduced inf-sup constant $\beta$. However, \cref{table_inf_sup} shows that $\beta$ is close to one for our numerical experiments. Therefore, we conjecture that in general it is possible to construct suitable local reduced spaces without knowledge about the global computational domain $\Omega$. Furthermore, if we assume that a heat source $f\in L^2(I,V^*)$ with $f\notin L^2(I,L^2(\Omega))$ is given, additional orthogonality assumptions are required to estimate the sum of local norms of $f$ in terms of the global norm of $f$ (cf. \cite[Chapter 4]{B19}).
\end{remark}

\subsection{Computational realization of the global GFEM approximation}
\label{implementation_global}

In the following we outline how a global GFEM approximation may be computed numerically. To simplify the notation we assume homogeneous initial and Dirichlet boundary conditions. Following the notation in \cref{implementation}, $V_h\subseteq H^1_0(\Omega)$ denotes a conforming FE space and the approximation $u_h\in S_t \otimes V_h$ of the solution is computed by solving 
\begin{equation*}
	((u_h)_t, v_h)_{L^2(I\times\Omega)} + (\alpha \nabla u_h, \nabla v_h)_{L^2(I\times\Omega)} = (f,v_h)_{L^2(I\times\Omega)} \quad \forall v_h \in Q_t \otimes V_h.
\end{equation*}
Well-posedness of the corresponding continuous problem has, e.g., been shown in \cite{SS09} by proving the inf-sup condition and surjectivity of the operator associated with the bilinear form. To compute an approximation of the global GFEM solution $u_\mathrm{GFEM}$, we assume that for each local subdomain $I\times\Omega^{in}_i$, $i=1,\ldots,M$, a reduced approximation space $\spann\lb \chi^i_1,\ldots, \chi^i_{N_i^{red}} \rb \subseteq S_t\otimes V_h\vert_{\Omega_i^{in}}$ of dimension $N_i^{red}$ was computed using either Algorithm 1 in subsection SM2.1 (deterministic approach based on a Krylov subspace method) or Algorithm 2 in subsection SM2.2 (randomized approach).
We choose partition of unity hat functions $\phi_{PU}^1,\ldots,\phi_{PU}^M$ as the basis functions of the linear FE space associated with the decomposition $\Omega = \cup_{i=1}^M \Omega^{in}_i$. Then, for every local reduced basis $\chi^i_1,\ldots, \chi^i_{N_i^{red}} $ and corresponding hat function $\phi_{PU}^i$ we compute their (discrete) pointwise product denoted by $ \chi^{i,PU}_1,\ldots,\chi^{i,PU}_{N_i^{red}}\in S_t \otimes V_{h,i,0}^{in}$, where $V_{h,i,0}^{in}:=\lb w \in V_h\vert_{\Omega_i^{in}} \mid w = 0 \text{ on } \partial\Omega_i^{in} \rb $. To define a test space for the Petrov-Galerkin $\mathrm{GFEM}$ approximation, we generate test functions $\varphi^i_1,\ldots, \varphi^i_{N_i^{red}}\in Q_t \otimes V_{h,i,0}^{in}$ by computing projections of the reduced ansatz functions $\chi^{i,PU}_1,\ldots,\chi^{i,PU}_{N_i^{red}}$ as follows: Find $\varphi_j^i\in Q_t \otimes V_{h,i,0}^{in}$, $1 \leq j \leq N_i^{red}$, such that
\begin{align*}
\begin{split}
((w_h)_t&,\varphi_j^i)_{L^2(I\times\Omega_i^{in})} + (\alpha\nabla w_h, \nabla \varphi_j^i)_{L^2(I\times\Omega_i^{in})}\\
& = ((w_h)_t,\chi^{i,PU}_j)_{L^2(I\times\Omega_i^{in})} + (\alpha\nabla w_h, \nabla \chi^{i,PU}_j)_{L^2(I\times\Omega_i^{in})}\quad \forall\, w_h \in S_t \otimes V_{h,i,0}^{in}.
\end{split}
\end{align*}
The ansatz and test space for the Petrov-Galerkin GFEM approximation are given by $X_h^{\mathrm{GFEM}} :=\spann\lb
\chi^{1,PU}_1,\ldots,\chi^{1,PU}_{N_1^{red}}, \chi^{2,PU}_1,\ldots,\chi^{2,PU}_{N_2^{red}}, \ldots, \chi^{M,PU}_1,\ldots,\chi^{M,PU}_{N_M^{red}}\rb $ and $V_h^\mathrm{GFEM} := \spann\lb\varphi^1_1,\ldots,\varphi^1_{N_1^{red}},\varphi^2_1,\ldots,\varphi^2_{N_2^{red}},\ldots,\varphi^M_1,\ldots,\varphi^M_{N_M^{red}} \rb$, respectively.\footnote{For this choice of $V_h^\mathrm{GFEM}$ we cannot guarantee stability in general, but \cref{table_inf_sup} shows that the inf-sup constant $\beta$ is close to one for our numerical experiments. However, a well-known strategy to ensure stability is to construct optimally stable pairs of trial and test spaces (see, for instance, \cite{BS19,DH12}).} Finally, we define $u_h^\mathrm{GFEM} \in X_h^\mathrm{GFEM}$ as the solution of the projected problem
\begin{align*}
((u_h^\mathrm{GFEM})_t,v_h)_{L^2(I\times\Omega)} + (\alpha \nabla u_h^\mathrm{GFEM}, \nabla v_h)_{L^2(I\times\Omega)} = (f,v_h)_{L^2(I\times\Omega)} \quad \forall\, v_h \in V_h^\mathrm{GFEM}.
\end{align*}

\begin{remark}[Computational complexity]
	Algorithm 3 in section SM3 summarizes the global GFEM approximation using either deterministic or randomized local basis generation. The algorithm enables a localized construction of the local ansatz spaces such that the global GFEM approximation satisfies a desired global error tolerance that is prescribed by the user.
	
	To compute a global approximation, we first construct $M$ local approximation spaces, which has a computational complexity of $\mathcal{O}(N_i^2 + (m_i+N_i^{red}+2) N_i^{4/3})$ in the deterministic case and $\mathcal{O}(N_i^2 + (n_t+N_i^{red}+2) N_i^{4/3})$ in the randomized setting (cf. \cref{remark_complexity,remark_comp_complexity} and subsections SM2.3 and SM3.1) for $i=1,\ldots,M$. Here, $N_i$ denotes the number of degrees of freedom on the oversampling subdomain $I\times\Omega_i^{out}$, $m_i\geq N_i^{red}$ is the number of eigenvalues and eigenvectors computed via a Krylov subspace method, and $n_t$ denotes the number of test vectors used in the randomized approach. We highlight that the computations for the local bases construction are embarrassingly parallel. If we employ random sampling to generate the local bases, also the computations for each subdomain can be parallelized.
	
	The reduced global system of size $N_{gl}^{red}\times N_{gl}^{red}$, where $N_{gl}^{red}:=\sum_{i=1}^{M}N_i^{red}$, has a block sparsity pattern resulting from the overlap of neighboring local subdomains and is dense within each block. To solve the global system one can employ, for instance, a preconditioned conjugate gradient method. The computational complexity for each iteration of the conjugate gradient method depends on the number of non-zero entries in the global reduced system matrix, which scales quadratically in the dimension of the local reduced spaces and linearly in the number of local subdomains, and is thus less than $\mathcal{O}((N_{gl}^{red})^2)$, but more than $\mathcal{O}(N_{gl}^{red})$ (for more details see subsection SM3.1). As the matrix resulting from the global GFEM approximation may have a large condition number, using a preconditioner is often mandatory. For more details on the computational complexity of the global approximation we refer to subsection SM3.1.
	
	Moreover, \cite{B19} provides a numerical experiment concerning the signal integrity simulation in a high frequency printed circuit board, where the dimension of the global FEM space is approximately $65$ million and a supercomputer would thus be required to solve the global system. A localized approach as suggested here allows to tackle such problems using common computer architectures and enables to parallel local computations.
\end{remark}


\section{Numerical Experiments}
\label{numerical_experiments}

In this section we numerically analyze the approximation properties of the (quasi-)\\optimal local reduced spaces and the global GFEM approximation introduced in \cref{local_appr_spaces,impl_and_random,global_approximation} and demonstrate that our approach is capable of approximating problems that include coefficients that are rough with respect to both space and time. In \cref{local_numerics} we focus on the local transfer eigenvalue problem and show the rapid decay of the eigenvalues for two model problems with high contrast and multiscale structure with respect to space and time. Moreover, we construct global GFEM approximations from the local reduced spaces using randomization and consider two test cases in \cref{global_numerics}, where one includes high contrast regarding the spatial and temporal variables. We demonstrate that the local approximation qualities carry over to the global approximation and that we can achieve a desired global error tolerance only by local computations. In all numerical experiments we equip the spaces of traces on the local oversampling boundaries with the corresponding $L^2$-inner products weighted with the diffusion coefficient $\alpha$. The complete source code to reproduce all results shown in this section is provided in \cite{code}.

\subsection{Analysis of the transfer eigenvalue problem}
\label{local_numerics}

\begin{figure}
	\begin{minipage}{0.49\textwidth}
		\centering
		\begin{tikzpicture} 
		\begin{semilogyaxis}[
		width=0.78\textwidth,
		height=4.5cm,
		xmin=0,
		xmax=1000,
		ymin=2e-4,
		ymax=2e0,
		xlabel=local basis size $n$,
		ylabel= $\sqrt{\lambda_{n+1}}$,
		grid=both,
		grid style={line width=.1pt, draw=gray!20},
		major grid style={line width=.2pt,draw=gray!50},
		xtick = {200,400,600,800},
		ytick={1e-3,1e-2,1e-1,1e0},
		minor xtick = {100,200,300,400,500,600,700,800,900,1000},
		minor ytick={1e-3,1e-2,1e-1,1e0},
		legend pos= outer north east,
		legend style = {font=\footnotesize},
		label style={font=\footnotesize},
		tick label style={font=\footnotesize}  
		]
		\addplot+[ thick, mark=x, mark repeat = 100] table[x expr=\coordindex,y index=3]{local_svd_0-3channels_1layer.dat};
		\addplot+[ thick, mark=triangle, mark repeat = 100] table[x expr=\coordindex,y index=2]{local_svd_0-3channels_1layer.dat};
		\addplot+[thick, mark=o, mark repeat = 100] table[x expr=\coordindex,y index=1]{local_svd_0-3channels_1layer.dat};
		\addplot+[thick, mark=none] table[x expr=\coordindex,y index=0]{local_svd_0-3channels_1layer.dat};
		\draw[domain=200:800,thick, densely dotted, variable = \x] plot ({\x}, {0.3*exp(-0.0085*\x)}) node[below left,rotate=-38,yshift=0.075cm] {\footnotesize$e^{-0.0085 n}$};
		\legend{$3$, $2$, $1$, $0$}
		\end{semilogyaxis}
		\end{tikzpicture}
		\captionsetup{width=0.95\linewidth}
		\captionof{figure}{\footnotesize
			Singular value decay for Example 1: $0$, $1$, $2$, or $3$ high conductivity channels and $1$ layer of oversampling.
		}
		\label{local_channels_comparison}
	\end{minipage}
	\hfill
	\begin{minipage}{0.49\textwidth}
		\centering
		\begin{tikzpicture} 
		\begin{semilogyaxis}[
		width=0.78\textwidth,
		height=4.5cm,
		xmin=0,
		xmax= 1000,
		ymin=1e-6,
		ymax=5e0,
		xlabel=local basis size $n$,
		ylabel= $\sqrt{\lambda_{n+1}}$,
		grid=both,
		grid style={line width=.1pt, draw=gray!20},
		major grid style={line width=.2pt,draw=gray!50},
		xtick = {200,400,600,800},
		ytick={1e-6,1e-4,1e-2,1e0},
		minor xtick = {100,200,300,400,500,600,700,800,900,1000},
		minor ytick={1e-6,1e-5,1e-4,1e-3,1e-2,1e-1,1e0},
		legend pos= outer north east,
		legend style = {font=\footnotesize},
		label style={font=\footnotesize},
		tick label style={font=\footnotesize}  
		]
		\addplot+[thick, mark=x, mark repeat = 100] table[x expr=\coordindex,y index=0]{local_svd_1channel_diff_layers.dat};
		\addplot+[thick, mark=triangle, mark repeat = 100] table[x expr=\coordindex,y index=1]{local_svd_1channel_diff_layers.dat};
		\addplot+[thick, mark=o, mark repeat = 100] table[x expr=\coordindex,y index=2]{local_svd_1channel_diff_layers.dat};
		\addplot+[thick, mark=none] table[x expr=\coordindex,y index=3]{local_svd_1channel_diff_layers.dat};
		\draw[domain=200:800,thick, densely dotted, variable = \x] plot ({\x}, {0.1*exp(-0.013*\x)}) node[below left,rotate=-35,yshift=0.075cm] {\footnotesize$e^{-0.013 n}$};
		\draw[domain=350:950,thick, densely dotted, variable = \x] plot ({\x}, {10*exp(-0.0044*\x)}) node[above left,rotate=-12,yshift=-0.035cm] {\footnotesize$e^{-0.0044 n}$};
		\legend{$0.5$,$1$, $1.5$, $2$}
		\end{semilogyaxis}
		\end{tikzpicture}
		\captionsetup{width=0.95\linewidth}
		\captionof{figure}{\footnotesize
			Singular value decay for Example 1: $1$ high conductivity channel and $0.5$, $1$, $1.5$, or $2$ layers of oversampling.
		}
		\label{local_channels_oversampling}
	\end{minipage}
\end{figure}

\begin{figure}
	\centering 
	\includegraphics[width=\textwidth]{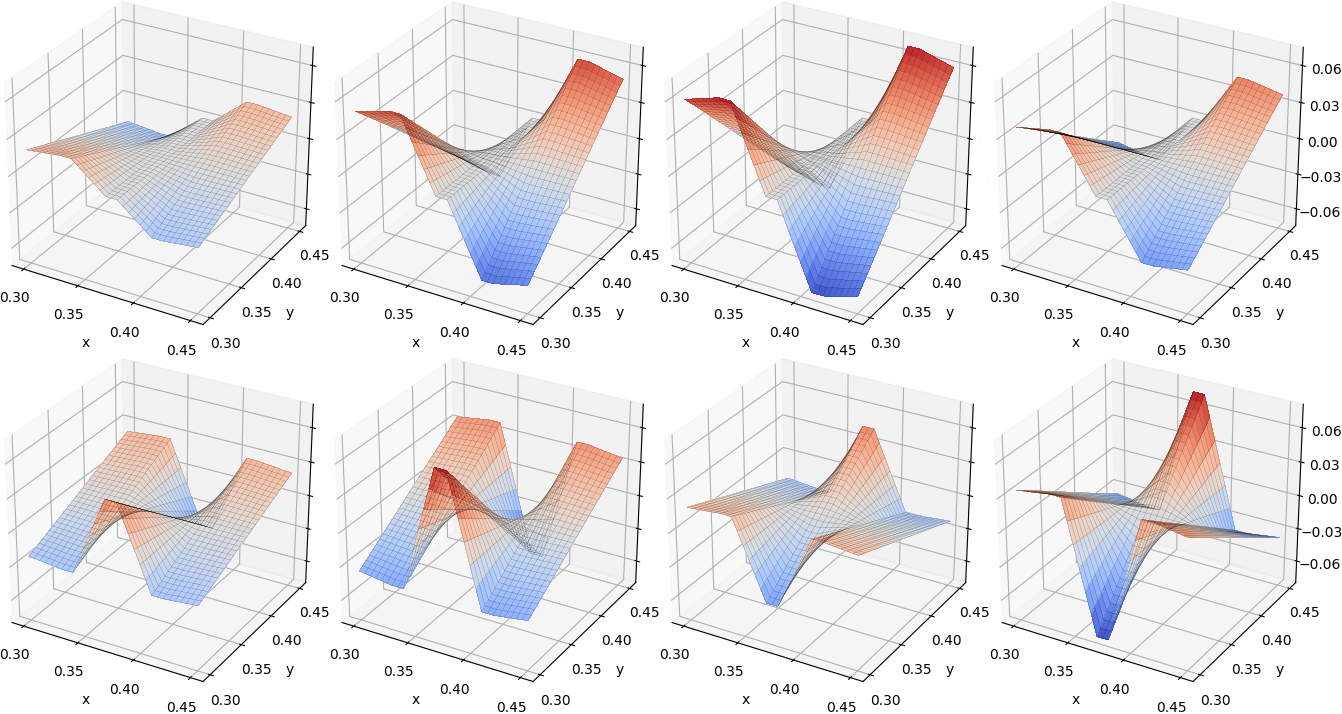}
	\captionsetup{width=0.98\linewidth}
	\captionof{figure}{\footnotesize Eigenfunctions corresponding to $\sqrt{\lambda_{50}}$ (top) and $\sqrt{\lambda_{101}}$ (bottom) on $\Omega^{in}$ at time $t=0.1$, $t=0.3$, $t=0.6$, and $t=0.8$ (left to right) for Example 1 with 3 high conductivity channels and 1 layer of oversampling.}	
	\label{local_channel_eigenfuncs}	
\end{figure}

To analyze the transfer eigenvalue problem \cref{EW_Transfer_int}, we first consider a test case including high conductivity channels with high contrast and refer to this numerical example as Example 1. We choose $I=(0,1)$, $\Omega^{in}=(0.3,0.45)^2$, and study different sizes of oversampling, ranging from $0.5$ to $2$ layers; in detail $\Omega^{out} \in \lb (0.225,0.525)^2, (0.15,0.6)^2 ,(0.075,0.675)^2 , (0,0.75)^2 \rb $. We discretize the local domains with a regular quadrilateral mesh with mesh size $1/200$ in both directions and the time interval by $50$ equidistant time steps. For numerical accuracy we employ in this subsection an implicit Euler time stepping. Furthermore, we consider zero to three high conductivity channels positioned at $I \times((0.33,0.34)\cup(0.37,0.38)\cup(0.41,0.42))\times(0,0.75)$, where $\alpha(t,x,y)=10^3$ in the channels and $\alpha(t,x,y)=1$ else. \cref{local_channels_comparison} shows that the singular values of the transfer operator first have a plateau and then decay exponentially, where the plateau is longer for a larger number of channels. This is due to the fact that the eigenfunctions corresponding to the singular values located in the plateau contain most energy in the local domain $\Omega^{in}$ and one can observe variations within the channels as exemplarily shown in \cref{local_channel_eigenfuncs}. The more channels, the more variations are possible. However, the exponential decay of the singular values is faster if the plateau is longer as already observed for the Helmholtz equation in \cite{BS18}. Furthermore, we see in \cref{local_channels_oversampling} that the singular values decay faster if we increase the number of oversampling layers, where simultaneously the computational costs increase. However, if we compare the  eigenfunctions corresponding to different oversampling sizes we can observe that these show very similar dynamics. We therefore conjecture that less oversampling is already sufficient to extract the significant modes.

\begin{figure}
	\begin{minipage}{0.385\textwidth}
		\centering
		\begin{tikzpicture} 
		\begin{semilogyaxis}[
		width=\textwidth,
		height=4.5cm,
		xmin=0,
		xmax=1000,
		ymin=2e-5,
		ymax=1e0,
		xlabel=local basis size $n$,
		ylabel= $\sqrt{\lambda_{n+1}}$,
		grid=both,
		grid style={line width=.1pt, draw=gray!20},
		major grid style={line width=.2pt,draw=gray!50},
		xtick = {200,400,600,800},
		ytick={1e-4,1e-3,1e-2,1e-1,1e0},
		minor xtick = {100,200,300,400,500,600,700,800,900,1000},
		minor ytick={1e-3,1e-2,1e-1,1e0},
		legend style={at={(0.98,0.98)},anchor=north east,font=\footnotesize},
		label style={font=\footnotesize},
		tick label style={font=\footnotesize}  
		]
		\addplot+[thick,  black , mark = none] table[x expr=\coordindex,y index=0]{local_eps_1_0.01.dat};
		\addplot+[densely dashed, red, thick, mark=x, mark repeat = 100, mark options = {solid}] table[x expr=\coordindex,y index=1]{local_eps_1_0.01.dat};
		\draw[domain=200:800,thick, densely dotted, variable = \x] plot ({\x}, {0.2*exp(-0.011*\x)}) node[below left,rotate=-40,yshift=0.075cm] {\footnotesize$e^{-0.011 n}$};
		\legend{$\varepsilon = 1$, $\varepsilon = 0.01$}
		\end{semilogyaxis}
		\end{tikzpicture}
		\captionsetup{width=0.95\linewidth}
		\captionof{figure}{\footnotesize
			Singular value decay for Example 2: $\varepsilon = 1$ vs. $\varepsilon=0.01$.
		}
		\label{local_multiscale}
	\end{minipage}
	\hfill
	\begin{minipage}{0.605\textwidth}
		\centering
		\vspace{0.1cm}
		\includegraphics[width=0.99\textwidth]{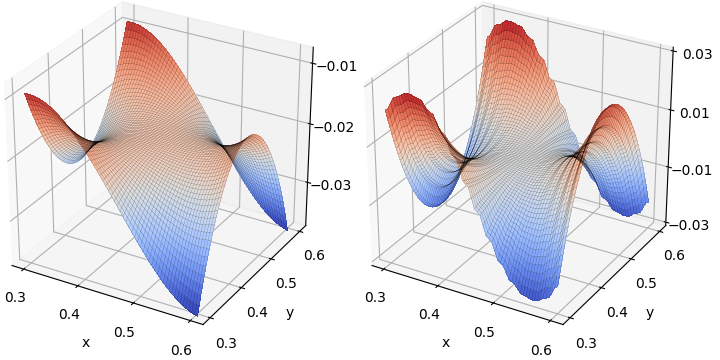}
		\captionsetup{width=0.99\linewidth}
		\captionof{figure}{\footnotesize Eigenfunctions corresponding to $\sqrt{\lambda_{200}}$ on $\Omega^{in}$ at time $t\mspace{-4mu}=\mspace{-4mu}0.3$ for $\varepsilon\mspace{-4mu} =\mspace{-4mu} 1$(left) vs. $\varepsilon\mspace{-4mu}=\mspace{-4mu}0.01$(right) for Example 2.}	
		\label{local_multiscale_eigenfunction}	
	\end{minipage}
\end{figure}

Next, we investigate a model problem including multiscale structure with respect to space and time, which we denote as Example 2. We consider $I=(0,0.4)$, $\Omega^{in} = (0.3,0.6)^2$, and $\Omega^{out} = (0,0.9)^2$ (1 layer of oversampling). Moreover, we use a regular quadrilateral mesh with mesh size $1/200$ in both directions for the spatial discretization and a step size of $1/100$ for the implicit Euler time stepping. We consider $\alpha_\varepsilon(t,x,y)= 10 + 8 \cos(\pi x/\varepsilon) + \cos(\pi t/\varepsilon)$ and compare the cases $\varepsilon = 1$ (only coarse scale) and $\varepsilon = 0.01$ (additionally including fine scale). In \cref{local_multiscale} we observe that the rapid singular value decay is almost identical for $\varepsilon = 1$ and $\varepsilon = 0.01$. Therefore, we can conclude that at least for this test case we do not require a higher number of local reduced basis functions for the approximation of the more complex problem including multiscale behavior regarding space and time. Finally, \cref{local_multiscale_eigenfunction} exemplarily shows that, as expected, the eigenfunctions inherit the multiscale structure of the problem.

\subsection{Global GFEM approximation with random local basis generation}
\label{global_numerics}

In this subsection we analyze the global GFEM approximation constructed from local reduced spaces that were generated using randomization to enable a more efficient computation; in detail we used Algorithms 2 and 3 in subsection SM2.2 and section SM3.\footnote{Following the notation in subsection SM2.2 and section SM3 we use a local and global failure probability of $\epsalgo = \epsfail = 10^{-15}$ and $n_t = 20$ test vectors. For an intense study on how the results of the randomized algorithm depend on parameters such as the number of test vectors see \cite{BS18}.} Throughout the whole subsection we consider the time interval $I = (0,0.5)$ and the global spatial domain $\Omega = (0,5)^2$ decomposed into $4\times 4$ local domains of size $2\times2$ with an overlap of size $1$. We also use an oversampling size of $1$. To ensure that the CFL condition for the numerical accuracy of the space-time Petrov-Galerkin discretization is satisfied (cf. \cite{A12}), we discretize the computational domain either with a regular quadrilateral mesh with mesh size $1/10$ in both directions and time step size $1/400$ (discretization 1) or with mesh size $1/15$ and time step size $1/900$ (discretization 2).

\begin{figure}
	\begin{minipage}{0.28\textwidth}
		\vspace{0.65cm}
		\centering 
		\includegraphics[width=0.88\textwidth]{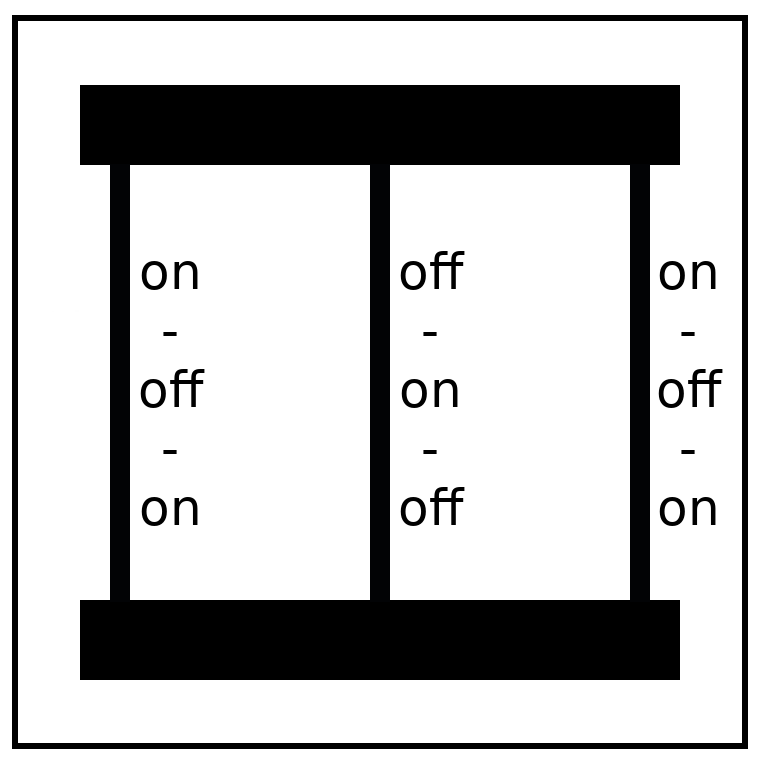}
		\captionsetup{width=0.99\linewidth}
		\captionof{figure}{\footnotesize Coefficient field for Example 4. White equates to $10^{-2}$ and black to $1$ (when channel is on) or $10^{-2}$ (when channel is off).}	
		\label{global_channel_coeff}	
	\end{minipage}
	\hfill
	\begin{minipage}{0.68\textwidth}
		\centering
		\begin{tikzpicture}
		\begin{loglogaxis}[
		title={Example 3.1},
		width=0.59\textwidth,
		height=4.5cm,
		xmin=0.7,
		xmax=1000,
		ymin=1e-6,
		ymax=5e-1,
		xlabel=local basis size $n$,
		grid=both,
		grid style={line width=.1pt, draw=gray!20},
		major grid style={line width=.2pt,draw=gray!50},
		xtick = {1e0,1e1,1e2,1e3},
		ytick={1e-6,1e-5,1e-4,1e-3,1e-2,1e-1},
		minor xtick = {1e0,1e1,1e2,1e3},
		minor ytick={1e-6,1e-5,1e-4,1e-3,1e-2,1e-1},
		legend pos= south west,
		legend style = {font=\footnotesize},
		label style={font=\footnotesize},
		tick label style={font=\footnotesize}  
		]
		\addplot+[thick, mark=none] table[x expr=\coordindex,y index=0]{slowest_svd_ex_1.dat};
		\addplot+[solid,          black, thick, mark=triangle, mark repeat=1] table[x index=0, y index=1] {local_global_ex_1.dat};
		\addplot+[densely dashed, red, thick, mark=x, mark repeat=1, mark options={solid}] table[x index=0, y index=2] {local_global_ex_1.dat};
		\legend{svd,local,global}
		\end{loglogaxis}
		\end{tikzpicture}
		\hspace{-0.5cm}
		\begin{tikzpicture}
		\begin{loglogaxis}[
		title={Example 4},
		width=0.59\textwidth,
		height=4.5cm,
		xmin=0.7,
		xmax= 200,
		ymin=1e-6,
		ymax=5e-1,
		xlabel=local basis size $n$,
		yticklabels={,,},
		grid=both,
		grid style={line width=.1pt, draw=gray!20},
		major grid style={line width=.2pt,draw=gray!50},
		xtick = {1e0,1e1,1e2},
		ytick={1e-6,1e-5,1e-4,1e-3,1e-2,1e-1},
		minor xtick = {1e0,1e1,1e2},
		minor ytick={1e-6,1e-5,1e-4,1e-3,1e-2,1e-1},
		label style={font=\footnotesize},
		tick label style={font=\footnotesize}  
		]
		\addplot+[solid,          black, thick, mark=triangle, mark repeat=1] table[x index=0, y index=1] {local_global_ex_2.dat};
		\addplot+[densely dashed, red, thick, mark=x, mark repeat=1, mark options={solid}] table[x index=0, y index=2] {local_global_ex_2.dat};
		\end{loglogaxis}
		\end{tikzpicture}
		\captionsetup{width=0.95\linewidth}
		\captionof{figure}{\footnotesize
			The slowest singular value decay, the maximum relative local error, and the relative global error. Median values over 20 realizations.
		}
		\label{global_basis_size}
	\end{minipage}
\end{figure}

We analyze two test cases: First, we consider a constant coefficient $\alpha\equiv 1$ with right hand side $f(t,x,y)=[\pi \cos(\pi t) + 2\pi^2/25 \sin(\pi t)] \sin(\pi x /5) \sin(\pi y /5)$ and homogeneous initial and Dirichlet boundary conditions. Hence, the analytical solution $u(t,x,y) = \sin(\pi t) \sin(\pi x /5 ) \sin(\pi y /5)$ is known. We refer to this numerical example as Example 3.1 (discretization 1) and 3.2 (discretization 2) depending on which discretization is used. Secondly, we investigate a model problem including high contrast with respect to space and time in terms of high conductivity channels that are switched on and off over time (Example 4 in the following). To this end, we define a heating region $\Gamma_{heat}:=I\times(0.4,4.6)\times(4,4.6)$, a cooling region $\Gamma_{cool}:= I\times(0.4,4.6)\times (0.4,1)$, and a channel region $\Gamma_{channel} := [((0,0.2)\cup (0.35,0.5)) \times (0.6,0.8)\times (1,4)] \cup  [(0.15,0.45)\times (2.4,2.6)\times(1,4)] \cup [((0,0.2)\cup (0.35,0.5))\times (4.2,4.4)\times (1,4) ]$ as depicted in \cref{global_channel_coeff}. Note that the channels are switched on and off at certain time points. We assume that $\alpha(t,x,y)=1$ for $(t,x,y)\in \Gamma_{heat}\cup \Gamma_{cool} \cup \Gamma_{channel}$, $\alpha(t,x,y)=10^{-2}$ else, $f(t,x,y)=1$ for $(t,x,y)\in \Gamma_{heat}$, $f(t,x,y)=-1$ for $(t,x,y)\in \Gamma_{cool}$, and $f(t,x,y)= 0$ else. Moreover, we prescribe homogeneous initial and Dirichlet boundary conditions and use discretization 2.

\begin{table}
	\centering
	\scriptsize
	\begin{tabular}{|c|c|c|c|c|c|c|c|c|c|c|c|}
		\hline
		local basis size $n$& 
		10  &
		25  & 
		50  & 
		100 &
		200 &
		300 &
		400 &
		500 &
		600 &
		700 &
		800\\
		\hline
		inf-sup constant $\beta$ 
		&  0.99
		&  0.99
		&  0.99
		&  0.99
		&  0.99
		&  0.99
		&  0.99
		&  0.98
		&  0.95
		&  0.91
		&  0.87 \\
		\hline
	\end{tabular}
	\captionsetup{width=0.9\linewidth}
	\vspace{0.1cm}
	\captionof{table}{\footnotesize Global reduced inf-sup constant $\beta$ exemplary for one realization of Example 3.1.}
	\label{table_inf_sup}
	\vspace{-0.2cm}
\end{table}

\begin{figure}
	\centering
	\begin{tikzpicture}
	\begin{loglogaxis}[
	title={Example 3.1},
	width=0.4\textwidth,
	height=4.5cm,
	xmin=190,
	xmax=850,
	ymin=3.5e-6,
	ymax=2.2e-4,
	xlabel=local basis size $n$,
	grid=both,
	grid style={line width=.1pt, draw=gray!20},
	major grid style={line width=.2pt,draw=gray!50},
	xtick = {200,300,400,500,600,700,800},
	xticklabels = {200,300,400, ,600, ,800},
	ytick={1e-6,1e-5,1e-4,1e-3},
	legend pos= outer north east,
	legend style = {font=\footnotesize},
	label style={font=\footnotesize},
	tick label style={font=\footnotesize}  
	]
	\addplot+[densely dotted,  black, thick, mark=none] table[x index=0, y index=3] {local_global_ex_1_cutout.dat};
	\addplot+[solid,  black, thick, mark=triangle, mark repeat=1] table[x index=0, y index=2] {local_global_ex_1_cutout.dat};
	\addplot+[densely dotted,  black, thick, mark=none] table[x index=0, y index=1] {local_global_ex_1_cutout.dat};
	\addplot+[dotted,  blue, thick, mark=none] table[x index=0, y index=6] {local_global_ex_1_cutout.dat};
	\addplot+[dashed,  red, thick, mark=x, mark repeat=1,mark options=solid] table[x index=0, y index=5] {local_global_ex_1_cutout.dat};
	\addplot+[dotted,  blue, thick, mark=none] table[x index=0, y index=4] {local_global_ex_1_cutout.dat};
	\legend{local max, local median, local min, global max, global median, global min}
	\end{loglogaxis}
	\end{tikzpicture}
	\begin{tikzpicture}
	\begin{loglogaxis}[
	title={Example 3.2},
	width=0.4\textwidth,
	height=4.5cm,
	xmin=0.7,
	xmax=1000,
	ymin=1e-6,
	ymax=2e-1,
	xlabel=local basis size $n$,
	grid=both,
	grid style={line width=.1pt, draw=gray!20},
	major grid style={line width=.2pt,draw=gray!50},
	xtick = {1e0,1e1,1e2,1e3},
	ytick={1e-6,1e-5,1e-4,1e-3,1e-2,1e-1},
	minor xtick = {1e0,1e1,1e2,1e3},
	minor ytick={1e-6,1e-5,1e-4,1e-3,1e-2,1e-1},
	legend style={at={(1.2,1)},anchor=north east,font=\footnotesize},
	label style={font=\footnotesize},
	tick label style={font=\footnotesize}  
	]
	\addplot+[solid,          black, thick, mark=triangle, mark repeat=1] table[x index=0, y index=1] {local_global_ex_1_finer.dat};
	\addplot+[densely dashed, red, thick, mark=x, mark repeat=1, mark options={solid}] table[x index=0, y index=2] {local_global_ex_1_finer.dat};
	\legend{local,global}
	\end{loglogaxis}
	\end{tikzpicture}
	\captionsetup{width=0.98\linewidth}
	\captionof{figure}{\footnotesize
		The minimum, median, and maximum relative local and global error over 20 realizations for Example 3.1 (left) versus the maximum relative local error and the relative global error for one realization of Example 3.2 (right). 
	}
	\label{global_basis_size_cutout}
	
\end{figure}
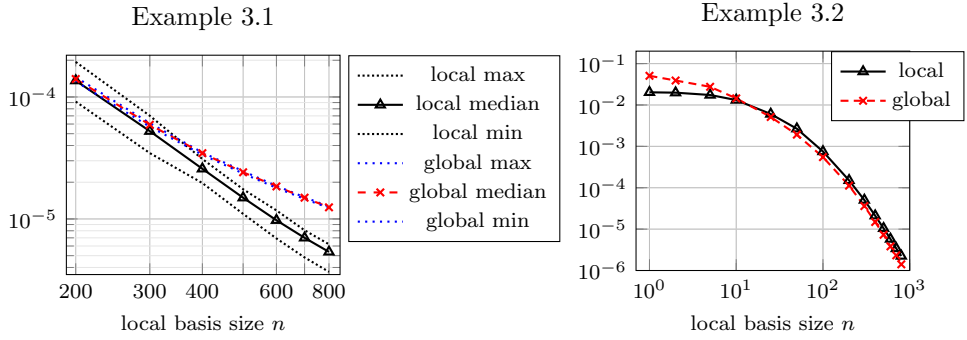

\begin{figure}
	\centering
	\begin{tikzpicture}
	\begin{loglogaxis}[
	title={Example 3.1},
	x dir=reverse,
	width=0.4\textwidth,
	height=4.5cm,
	xmin=7e-3,
	xmax=1.5e1,
	ymin=5e-7,
	ymax=2e-2,
	legend pos=outer north east,
	xlabel=$\varepsilon$,
	grid=both,
	grid style={line width=.1pt, draw=gray!20},
	major grid style={line width=.2pt,draw=gray!70},
	xtick={1e-2,1e-1,1e0, 1e1},
	ytick={1e-6,1e-5, 1e-4, 1e-3, 1e-2},
	minor xtick={1e-2,1e-1, 1e0, 1e1, 1e2},
	minor ytick={1e-6,1e-5, 1e-4, 1e-3, 1e-2},
	ylabel= relative local error,
	label style={font=\footnotesize},
	tick label style={font=\footnotesize}  
	]
	\addplot+[densely dotted, blue, thick,  mark=none] table[x index=0, y index=5] {local_tolerances_ex_1.dat};
	\addplot+[densely dashed, blue, thick,  mark=none] table[x index=0, y index=4] {local_tolerances_ex_1.dat};
	\addplot+[solid,          black, thick, mark=none] table[x index=0, y index=3] {local_tolerances_ex_1.dat};
	\addplot+[densely dashed, brown, thick, mark=none] table[x index=0, y index=2] {local_tolerances_ex_1.dat};
	\addplot+[densely dotted, brown, thick, mark=none] table[x index=0, y index=1] {local_tolerances_ex_1.dat};
	\end{loglogaxis}
	\end{tikzpicture}
	\begin{tikzpicture}
	\begin{loglogaxis}[
	title={Example 4},
	x dir=reverse,
	width=0.4\textwidth,
	height=4.5cm,
	xmin=1.6e-3,
	xmax=1.5e1,
	ymin=5e-7,
	ymax=2e-2,
	legend pos=outer north east,
	legend style = {font=\footnotesize},
	xlabel=$\varepsilon$,
	yticklabels={,,},
	grid=both,
	grid style={line width=.1pt, draw=gray!20},
	major grid style={line width=.2pt,draw=gray!70},
	xtick={1e-2,1e-1,1e0,1e1},
	ytick={1e-6,1e-5, 1e-4, 1e-3, 1e-2,1e-1},
	minor xtick={1e-2,1e-1,1e0,1e1},
	minor ytick={1e-6,1e-5, 1e-4, 1e-3, 1e-2,1e-1},
	label style={font=\footnotesize},
	tick label style={font=\footnotesize}  
	]
	\addplot+[densely dotted, blue, thick,  mark=none] table[x index=0, y index=5] {local_tolerances_ex_2.dat};
	\addplot+[densely dashed, blue, thick,  mark=none] table[x index=0, y index=4] {local_tolerances_ex_2.dat};
	\addplot+[solid,          black, thick, mark=none] table[x index=0, y index=3] {local_tolerances_ex_2.dat};
	\addplot+[densely dashed, brown, thick, mark=none] table[x index=0, y index=2] {local_tolerances_ex_2.dat};
	\addplot+[densely dotted, brown, thick, mark=none] table[x index=0, y index=1] {local_tolerances_ex_2.dat};
	\legend{max, 75 percentile, 50 percentile, 25 percentile, min};
	\end{loglogaxis}
	\end{tikzpicture}
	\captionsetup{width=0.98\linewidth}
	\captionof{figure}{\footnotesize Relative local error versus local error tolerance $\varepsilon$. Statistics over 20 samples and all 16 local spaces.
	}
	\label{global_tolerances_local}
\end{figure}
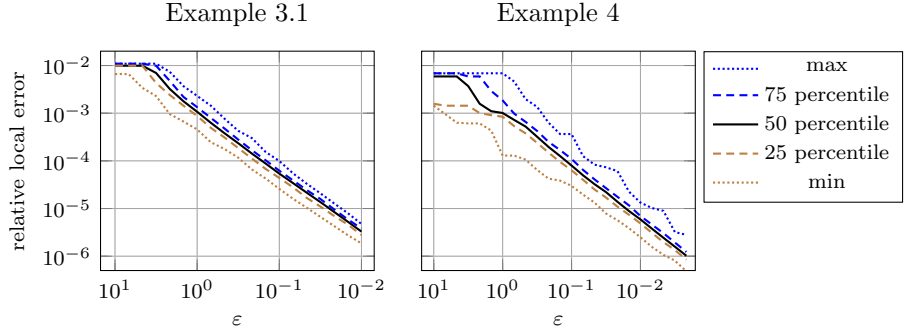

\begin{figure}
	\begin{minipage}{0.72\textwidth}	
		\centering
		\begin{tikzpicture}
		\begin{loglogaxis}[
		title={Example 3.1},
		x dir=reverse,
		width=0.53\textwidth,
		height=4.5cm,
		xmin=3.5e-1,
		xmax=7e2,
		ymin=1e-6,
		ymax=1e-1,
		legend pos=outer north east,
		xlabel=$\tolGFEM$,
		grid=both,
		grid style={line width=.1pt, draw=gray!20},
		major grid style={line width=.2pt,draw=gray!70},
		xtick={1e0, 1e1, 1e2},
		ytick={1e-6,1e-5, 1e-4, 1e-3, 1e-2,1e-1},
		minor xtick={ 1e0, 1e1, 1e2},
		minor ytick={1e-6,1e-5, 1e-4, 1e-3, 1e-2,1e-1},
		ylabel= relative global error,
		label style={font=\footnotesize},
		tick label style={font=\footnotesize}  
		]
		\addplot+[densely dotted, blue, thick,  mark=none] table[x index=0, y index=3] {global_tolerances_ex_1.dat};
		\addplot+[solid,          black, thick, mark=none] table[x index=0, y index=2] {global_tolerances_ex_1.dat};
		\addplot+[densely dotted, brown, thick, mark=none] table[x index=0, y index=1] {global_tolerances_ex_1.dat};
		\end{loglogaxis}
		\end{tikzpicture}
		\begin{tikzpicture}
		\begin{loglogaxis}[
		title={Example 4},
		x dir=reverse,
		width=0.53\textwidth,
		height=4.5cm,
		xmin=2.5e-2,
		xmax=5.5e2,
		ymin=1e-6,
		ymax=1e-1,
		legend style = {font=\footnotesize,at={(1.25,0.98)}},
		xlabel=$\tolGFEM$,
		yticklabels={,,},
		grid=both,
		grid style={line width=.1pt, draw=gray!20},
		major grid style={line width=.2pt,draw=gray!70},
		xtick={1e-1,1e0,1e1,1e2},
		ytick={1e-6,1e-5, 1e-4, 1e-3, 1e-2,1e-1},
		minor xtick={1e-1,1e0,1e1,1e2},
		minor ytick={1e-6,1e-5, 1e-4, 1e-3, 1e-2,1e-1},
		label style={font=\footnotesize},
		tick label style={font=\footnotesize}  
		]
		\addplot+[densely dotted, blue, thick,  mark=none] table[x index=0, y index=3] {global_tolerances_ex_2.dat};
		\addplot+[solid,          black, thick, mark=none] table[x index=0, y index=2] {global_tolerances_ex_2.dat};
		\addplot+[densely dotted, brown, thick, mark=none] table[x index=0, y index=1] {global_tolerances_ex_2.dat};
		\legend{max, 50 percentile, min};
		\end{loglogaxis}
		\end{tikzpicture}
		\captionsetup{width=0.96\linewidth}
		\captionof{figure}{\footnotesize Relative global error versus global error tolerance $\tolGFEM$. Statistics over 20 samples.
		}
		\label{global_tolerances}
	\end{minipage}
	\hfill
	\begin{minipage}{0.25\textwidth}
		\centering 
		\includegraphics[width=0.95\textwidth]{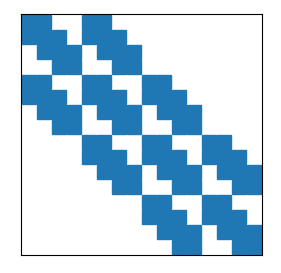}
		\captionsetup{width=0.98\linewidth}
		\captionof{figure}{\footnotesize Sparsity pattern of global reduced matrix.}	
		\label{sparsity_pattern}
	\end{minipage}
\end{figure}

While we cannot guarantee stability of the global $\mathrm{GFEM}$ approximation in general, we observe in \cref{table_inf_sup} that for Example 3.1 the inf-sup constant $\beta$ is close to one and we thus have stability. \cref{global_basis_size} shows that the global error\footnote{\scriptsize$\big(\Vert (u_h - u_h^\mathrm{GFEM})_t \Vert_{L^2(I\times\Omega)}^2 + \Vert \alpha^\frac{1}{2} \nabla (u_h - u_h^\mathrm{GFEM}) \Vert_{L^2(I\times\Omega)}^2 \big)^{1/2}/$ $ \big(\Vert \alpha^\frac{1}{2} \nabla u_h \Vert_{L^2(I\times\Omega)} + \Vert f\Vert_{L^2(I\times\Omega)}\big)$} convergence is clearly guided by the local error\footnote{\scriptsize$\min_{w_i^n \in \Lambda_i^{n,\text{data}}} \Vert \alpha^\frac{1}{2} \nabla (u_h\vert_{I\times\Omega_i^{in}} - w_i^n) \Vert_{L^2(I\times\Omega_i^{in})} / \big(\Vert \alpha^\frac{1}{2} \nabla u_h\vert_{I\times\Omega_i^{out}} \Vert_{L^2(I\times\Omega_i^{out})} + \Vert f\Vert_{L^2(I\times\Omega_i^{out})}\big)$} convergence, which is in turn very similar to the singular value decay. This is in line with the predictions by theory, cf. \cref{error_bounds}. Although we only use $0.5$ layers of oversampling for the computation of the local reduced spaces this seems to be sufficient to yield good approximation properties and extract the significant modes which confirms the conjecture in \cref{local_numerics} at least for the given numerical examples. If we have a closer look at the results for Example 3.1, we observe that for increasing local basis sizes $n\geq400$ the global error behavior seems to change as shown in \cref{global_basis_size_cutout} (left). However, this can be traced back to the coarse spatial discretization since the result for the finer discretization employed in Example 3.2 again shows that the global errors exactly follow the local errors (cf. \cref{global_basis_size_cutout} (right)). For the results shown in \cref{global_tolerances_local,global_tolerances} we used the adaptive randomized Algorithms 2 and 3 introduced in subsection SM2.2 and section SM3\footnote{Since we assume that $u_0=g_D=g_N=0$ for Example 3 and 4 (thus $u^b=0$), we used Algorithm 3 with the improved relative global error bound $ 2 \sqrt{ M^{out}} \, \max_{i=1,\ldots,M} \big\lbrace C_i(c_1,c_2,\beta,M^{in},\Omega_i^{in},c_{p,i}^{\alpha}) $ $\, \max\lb 1, c_{f,i}\rb\,\varepsilon_i \big \rbrace$ (cf. \cref{global_error}) to generate the results shown in \cref{global_tolerances}.} and chose $\beta$ equal to one due to the results observed in \cref{table_inf_sup}. Algorithm 3 enables us to prescribe a global error tolerance and generate an approximation that satisfies the desired tolerance with probability ($1-\epsfail$). In \cref{global_tolerances} we observe that in our numerical experiments the algorithm always succeeded and the same holds true for the outcome of Algorithm 2 (cf.  \cref{global_tolerances_local}). The local\footnote{\scriptsize$\min_{w_i^n \in \Lambda_i^{n,\text{data}}} \Vert \alpha^\frac{1}{2} \nabla (u_h - w_i^n) \Vert_{L^2(I\times\Omega_i^{in})} / \max\lb 2, c_{f,i}\rb \big(\Vert \alpha^\frac{1}{2} \nabla u_h \Vert_{L^2(I\times\Omega_i^{out})} + \Vert f\Vert_{L^2(I \times\Omega_i^{out})}\big)$} (global\footnote{\scriptsize$\big(\Vert (u_h - u_h^\mathrm{GFEM})_t \Vert_{L^2(I\times\Omega)}^2 + \Vert \alpha^\frac{1}{2} \nabla (u_h - u_h^\mathrm{GFEM}) \Vert_{L^2(I\times\Omega)}^2 \big)^{1/2}/$ $ \big(\Vert \alpha^\frac{1}{2} \nabla u_h \Vert_{L^2(I\times\Omega)} + \Vert f\Vert_{L^2(I\times\Omega)}\big)$}) results are more accurate than required by about $2$ to $3.5$ ($4$ to $4.5$) orders of magnitude. The former is due to the fact that Algorithm 2 constructs local spaces that have better approximation properties than required, where the latter is additionally caused by the pessimistic global a priori error bound (cf. \cref{global_error}) that is employed to calculate local error tolerances from the prescribed global error tolerance in Algorithm 3. Finally, \cref{sparsity_pattern} shows the sparsity pattern of the global reduced system matrix, where the block sparsity structure results from the overlap of neighboring local subdomains.


\section{Conclusions}
\label{conclusions}

We have proposed local space-time ansatz spaces for linear parabolic PDEs that are optimal in the sense of Kolmogorov \cite{K36} and can be used in domain decomposition and multiscale methods. The diffusion coefficient of the PDE may be arbitrarily rough with respect to both space and time. The optimal local spaces are spanned by the left singular values of a compact transfer operator that acts on the space of local solutions. Moreover, we have employed a space-time GFEM to couple the local ansatz spaces and construct an approximation of the global solution. Furthermore, we have derived rigorous local and global a priori error bounds. In particular, we have shown that the global approximation error in a suitable graph norm can be bounded only by the local approximation errors in the $L^2(H^1)$-seminorm. Finally, we have proposed an adaptive algorithm for the localized construction of the local ansatz spaces such that the global GFEM approximation satisfies a desired global error tolerance.

The numerical experiments demonstrate a very rapid and exponential decay of the singular values of the transfer operator and the local and global approximation errors for problems with high contrast or multiscale structure regarding both the spatial and the temporal variable. Since in addition no time stepping procedure is required, the computation of the global approximation is very efficient. For a multiscale test case we have observed that the singular values of the transfer operator seem to be independent of the size of the parameter $\varepsilon$ that determines the periodicity of fine-scale variations in the diffusion coefficient. We conjecture that this result transfers to $\varepsilon\rightarrow0$ at least for the considered test case. Moreover, we have observed that the global error convergence is clearly guided by the local error convergence as predicted by theory. Our numerical experiments have shown that the global reduced inf-sup constant is close to one and we thus have stability at least for the considered test cases.

For future applications it is favorable to investigate numerically more efficient discretization techniques such as low rank tensor formats (cf. \cite{GK13} and references therein).


\appendix

\section{Proofs}
\label{proofs}

\subsection{Proofs of \cref{interior}}
\label{proofs_int}

In the remainder of this subsection we denote the $L^2(\Omega^{in})$ inner product by $(\cdot\,,\cdot)_{in}:=(\cdot\,,\cdot)_{L^2(\Omega^{in})}$ and the dual pairing between $H^1_0(\Omega^{in})$ and its dual space $H^{-1}(\Omega^{in})$ by $\la\cdot\,,\cdot\ra_{in}:=\la \cdot\,, \cdot \ra_{H^1_0(\Omega^{in})}$. Analogously, we define $(\cdot\,,\cdot)_{out}:=(\cdot\,,\cdot)_{L^2(\Omega^{out})}$ and $\la\cdot\,,\cdot\ra_{out}:=\la \cdot\,, \cdot \ra_{H^1_0(\Omega^{out})}$.

\begin{proof}[Proof of \cref{par_cacc_int} (Parabolic Caccioppoli inequality)]
	Since $w$ is contained in $\HH^{out}$, we have that $\int_0^T \la w_t(t),v\ra_{out}\, \varphi(t)\; dt + \int_0^T (\alpha \nabla  w(t), \nabla v)_{out} \,\varphi(t) \;dt = 0$ for all $v \in H_0^1(\Omega^{out})$ and $\varphi \in C_0^\infty(I)$. As $\varphi \in C_0^\infty(I)$ is chosen arbitrarily the fundamental lemma of calculus of variations yields that for almost every $t\in (0,T)$
	\begin{align}
	\la w_t(t),v\ra_{out} + (\alpha \nabla w(t), \nabla v)_{out} = 0 \quad \forall\,v\in H^1_0(\Omega^{out}).
	\label{bew_cacc1_int}
	\end{align}
	Let $\eta \in C^1_0(\Omega^{out})$ denote a cut-off function with the following properties: $0 \leq \eta \leq 1$, $\eta=1$ in $\Omega^{in}$, and $|\nabla \eta| \leq \frac{1}{\delta}$. Additionally choose $t^*\in (0,T]$ arbitrarily. In the following we want to use $w \eta^2$ as a test function. However, we cannot apply partial integration in time since the function $w$ is contained in $L^2(I,H^1(\Omega^{out}))$, but not in $ L^2(I,H^1_0(\Omega^{out}))$. We therefore approximate $w\eta\in L^2((0,t^*),H^1_0(\Omega^{out}))$ by a sequence $\left(w_n\right)_{n\in\N}\subseteq C_0^\infty((0,t^*),H^1_0(\Omega^{out}))$ such that $w_n \xlongrightarrow{n \rightarrow \infty} w\eta \,\;\text{ in } \, L^2((0,t^*),H^1_0(\Omega^{out}))$ (to simplify the notation we omit the restriction $w\vert_{(0,t^*)\times \Omega^{out}}$). Then, we have that $w_n(t)\eta \in H^1_0(\Omega^{out})$ for almost every $t \in (0,t^*)$ and each $n\in \N$ and $ \la w_t(t),w_n(t)\eta\ra_{out}  + \la A_\alpha w(t),w_n(t)\eta\ra_{out} = 0$ for almost every $t \in (0,t^*)$ according to \cref{bew_cacc1_int}. Consequently, integrating over time yields $\int_0^{t^*} \la w_t(t),w_n(t)\eta\ra_{out}\, dt + \int_0^{t^*} \la A_\alpha w(t), w_n(t)\eta\ra_{out}\, dt = 0$ for each $n\in\N$. Since $\eta $ does not depend on time, we can infer that
	\begin{align*}
	\int_0^{t^*} \la w_t(t),w_n(t)\eta\ra_{out}  dt
	&= -\int_0^{t^*} ( w(t),\left(w_n(t)\eta\right)_t )_{out}  dt= -\int_0^{t^*} (w(t)\eta,\left(w_n(t)\right)_t)_{out} dt  \\
	&=\hspace{8pt} \int_0^{t^*} \la (w(t)\eta)_t,w_n(t)\ra_{out} dt \,= \int_0^{t^*} \la w_t(t)\eta,w_n(t)\ra_{out}  dt.
	\end{align*}
	This implies that $\int_0^{t^*} \la w_t(t)\eta,w_n(t)\ra_{out} \, dt + \int_0^{t^*} \int_{\Omega^{out}} \alpha \nabla w(t) \nabla(w_n(t)\eta)\;dt = 0$.	As $w_n \xlongrightarrow{n \rightarrow \infty} w\eta$ in $L^2((0,t^*),H^1_0(\Omega^{out}))$ the sequence especially converges weakly in $L^2((0,t^*),H^1_0(\Omega^{out}))$. Hence, we can conclude that 
	\begin{align*}
	\int_0^{t^*} \la w_t(t)\eta,w(t)\eta \ra_{out}\, dt + \int_0^{t^*} \int_{\Omega^{out}} \alpha \vert\nabla w(t)\vert^2 \eta^2 + 2\,\alpha \nabla w(t)\,w(t)\,\eta \nabla \eta \;dt = 0.
	\end{align*}
	Finally, we apply partial integration in time  and exploit $w(0)=0$ in $L^2(\Omega^{out})$ to obtain $\int_0^{t^*} \la w_t(t)\eta,w(t)\eta \ra_{out}\, dt = \frac{1}{2}\Vert w(t^*)\eta\Vert_{L^2(\Omega^{out})}^2 - \frac{1}{2}\Vert w(0)\eta\Vert_{L^2(\Omega^{out})}^2 = \frac{1}{2}\Vert w(t^*)\eta\Vert_{L^2(\Omega^{out})}^2$. By applying the Cauchy-Schwarz and Young's inequality we have that
	\begin{align*}
	\Vert w(t^*)\eta\Vert^2_{L^2(\Omega^{out})}+\int_0^{t^*} \int_{\Omega^{out}} \alpha \vert\nabla w(t)\vert^2\eta^2 dt	\leq 4 \int_0^{t^*}\int_{\Omega^{out}} \alpha\, w(t)^2 \vert\nabla \eta\vert^2 dt.
	\end{align*}
	By exploiting the properties of the cut-off function $\eta$ we infer that
	\begin{align}\label{bew_cacc3_int}
	\Vert w(t^*)\vert_{\Omega^{in}}\Vert^2_{L^2(\Omega^{in})}\mspace{-4mu}+\mspace{-4mu}\Vert \alpha^\frac{1}{2}\nabla (w\vert_{(0,t^*)\times\Omega^{in}})\Vert^2_{L^2((0,t^*)\times\Omega^{in})}\mspace{-4mu}\leq\mspace{-4mu} \frac{4\,\alpha_1}{\delta^2}\int_0^{T} \mspace{-4mu}\int_{\Omega^{out}}\mspace{-4mu} w(t)^2 dt.
	\end{align}
	Bounding each term in \cref{bew_cacc3_int} by $ \frac{4 \alpha_1}{\delta^2}\Vert w \Vert^2_{L^2(I\times\Omega^{out})}$ and taking in both inequalities the supremum over all $t^*\in (0,T]$ yields the estimate.
\end{proof}

\begin{proof}[Proof of \cref{harmonic_prop_int} (Regularity)]
	Let a function $w\in \HH^{in}$ be given. Employing partial integration in time, we have that
	\begin{align*}
	\int_I \la w_t(t)+A_{\alpha} w(t),v \varphi(t)\ra_{in}\,dt = 0\quad \forall\,v\in H^1_0(\Omega^{in}),\,\varphi \in C_0^\infty(I). 
	\end{align*}
	Here, the linear operator $A_\alpha\mspace{-2mu}:\mspace{-2mu}L^2(I,H^1(\Omega^{in}))\mspace{-2mu} \rightarrow\mspace{-2mu} (L^2(I,H^1_0(\Omega^{in})))^*\mspace{-2mu}\cong\mspace{-2mu} L^2(I,H^{-1}(\Omega^{in}))$, $(A_\alpha w)(t):= \tilde{A}_\alpha(w(t))$, is induced by the linear operator $\tilde{A}_\alpha:H^1(\Omega^{in}) \rightarrow H^{-1}(\Omega^{in})$, $\la \tilde{A}_\alpha w, v\ra_{in} := \int_{\Omega^{in}} \alpha \nabla w \nabla v$. Since the space $\lb v\varphi \mid v\in H^1_0(\Omega^{in}),\;\varphi\in C_0^\infty(I)\rb$ is a dense subspace of $L^2(I,H^1_0(\Omega^{in}))$, we apply the Hahn-Banach theorem to infer that
	\begin{align*}
	w_t + A_{\alpha}w = 0 \;\text{ in }  L^2(I,H^{-1}(\Omega^{in})).
	\end{align*}
	It follows that $w_t \in L^2(I,H^{-1}(\Omega^{in}))$. The result for $\HH^{out}$ can be shown analogously.
\end{proof}

\begin{theorem}[Compactness of transfer operator] \label{compact_op_int}
	The transfer operator $P:\BB^{out}\rightarrow\HH^{in}$ introduced in \cref{trans_op_int} is compact.
\end{theorem}

\begin{proof}
	Let $(\tilde{w}_k)_{k\in\N}\subseteq \BB^{out}$ denote a sequence that is bounded with respect to $\lenergy \cdot \renergy_{out}$. Consequently, the sequence $(w_k)_{k\in\N}:= (H(\tilde{w}_k))_{k\in\N} \subseteq \HH^{out}$ is bounded in $L^2(I,H^1(\Omega^{out}))$. Then, there exists a subsequence $(w_{k_l})_{l\in\N}$ that converges weakly to a limit function $w \in L^2(I,H^1(\Omega^{out}))$. Next, we show that $w\in\HH^{out}$. The weak convergence yields $-\int_I (w(t),v)_{out}\,\varphi_t(t)\,dt + \int_I  (\alpha \nabla w (t), \nabla v)_{out}\,\varphi(t)\,dt = 0$ since we have $w_{k_l}\in\HH^{out}$ for every $l\in\N$. Exploiting arguments completely analogous to the proof of \cref{harmonic_prop_int}, we  obtain $w_t \in L^2(I,H^{-1}(\Omega^{out}))$, thus $w \in W^{1,2,2}(I,H^1(\Omega^{out}),$ $H^{-1}(\Omega^{out}))$\footnote{Recall that $W^{1,2,2}(I,H^1(\Omega^{out}),H^{-1}(\Omega^{out}))\mspace{-2mu}:=\mspace{-2mu} \left\lb u \mspace{-2mu}\in\mspace{-2mu} L^2(I,H^1(\Omega^{out}))\mspace{-2mu} \mid \mspace{-2mu}u_t \mspace{-2mu}\in\mspace{-2mu} L^2(I,H^{-1}(\Omega^{out})) \right\rb$.} and thanks to the embedding $W^{1,2,2}(I,H^1(\Omega^{out}),H^{-1}(\Omega^{out}))\hookrightarrow C^0(\bar{I},$ $L^2(\Omega^{out}))$ we have $w\in L^\infty(I,L^2(\Omega^{out}))$. It remains to show that $w$ satisfies homogeneous initial conditions. For $v\in H^1_0(\Omega^{out})$ and $\varphi\in C_0^\infty(I)$ we infer that 
	\begin{align*}
	\int_I \la w_t(t),v\ra_{out}\varphi(t) dt &= - \int_I \left(w(t),v\right)_{out}\varphi_t(t)dt\\
	&= \underset{l\rightarrow\infty}{\text{lim}} \hspace{-2pt}- \hspace{-2pt}\int_I \left(w_{k_l}(t),v\right)_{out}\varphi_t(t)dt = \underset{l\rightarrow\infty}{\text{lim}}\int_I \la \left(w_{k_l}\right)_t(t),v\ra_{out}\varphi(t)dt.
	\end{align*}
	Exploiting the density of $\lb v\varphi \mid v\in H^1_0(\Omega^{out}),\;\varphi\in C_0^\infty(I)\rb$ in $L^2(I,H^1_0(\Omega^{out}))$ thus yields that $(\left(w_{k_l}\right)_t)_{l\in\N}$ converges weakly-* to $w_t$ in $L^2(I,H^{-1}(\Omega^{out}))$. Choosing test functions $v\in H^1_0(\Omega^{out})$ and $\varphi\in C^\infty(I)$ satisfying $\varphi(T)=0$ we can conclude 
	\begin{align*}
	\left(w(0),v\right)_{out}\,\varphi(0) &= -\int_I \left(w(t),v\right)_{out}\,\varphi_t(t)\,dt - \int_I \la w_t(t),v\ra_{out}\,\varphi(t)\,dt\\
	&=\underset{l\rightarrow\infty}{\text{lim}} -\int_I \left(w_{k_l}(t),v\right)_{out}\,\varphi_t(t)\,dt - \int_I \la \left(w_{k_l}\right)_t(t),v\ra_{out}\,\varphi(t)\,dt\\
	&=\underset{l\rightarrow\infty}{\text{lim}} \left(w_{k_l}(0),v\right)_{out}\,\varphi(0) =\underset{l\rightarrow\infty}{\text{lim}} \left(0,v\right)_{out}\,\varphi(0) =\,0.
	\end{align*}
	As $\varphi(0)$ is arbitrary and $H^1_0(\Omega^{out})$ is dense in $L^2(\Omega^{out})$ it follows that $w(0)=0$ in $L^2(\Omega^{out})$. Hence, the limit $w$ is an element of $\HH^{out}$. 
	
	As $(w_{k_l})_{l\in\N}$ is bounded in $W^{1,2,2}(I,H^1(\Omega^{out}),H^{-1}(\Omega^{out}))$ thanks to \cref{harmonic_prop_int}, the compactness theorem of Aubin-Lions \cite[Corollary 5]{S86} then yields a subsequence $(w_{k_{l_m}})_{m\in\N}$ which converges strongly to $w$ in $L^2(I,L^2(\Omega^{out}))$ due to the uniqueness of weak limits. Considering the error sequence $(e_{k_{l_m}})_{m\in\N}:=( w_{k_{l_m}}-w)_{m\in\N}\subseteq \HH^{out}$, we thus have that $e_{k_{l_m}}   \xlongrightarrow{m \rightarrow \infty} 0$ in $L^2(I,L^2(\Omega^{out}))$. Since each $e_{k_{l_m}}$ is contained in $\HH^{out}$, the parabolic Caccioppoli inequality \cref{caccioppoli_int} yields $\Vert \alpha^\frac{1}{2}\nabla (e_{k_{l_m}}\vert_{I\times\Omega^{in}}) \Vert^2_{L^2(I\times\Omega^{in})}\leq \frac{8\,\alpha_1}{\delta^2}\Vert e_{k_{l_m}} \Vert^2_{L^2(I\times\Omega^{out})}\xlongrightarrow{m \rightarrow \infty} 0$ and we obtain $w_{k_{l_m}}\vert_{I\times\Omega^{in}}  \xlongrightarrow{m \rightarrow \infty} w\vert_{I\times\Omega^{in}}$ in $\HH^{in}$.
\end{proof}

\begin{proposition}[Parabolic Poincar\'e inequality]\label{par_poin_int}
	There exists a constant $c_{p}^{in}>0$ such that for all functions $w\in\HH^{in}$
	\begin{align}
	\Vert w \Vert_{L^2(I,L^2(\Omega^{in}))}\leq c_p^{in}\,\Vert \nabla w \Vert_{L^2(I,L^2(\Omega^{in}))}.
	\label{poincare_int}
	\end{align}
	An analogous result holds for functions in $\HH^{out}$.
\end{proposition}

Similar parabolic Poincar\'e inequalities can, for instance, be found in \cite[Theorem 2.2]{A16}, \cite[Lemma 2.5]{K08}, and \cite[Lemma 3]{S81}. However, the idea of the proof of inequality \cref{poincare_int} is guided by the proof of the (elliptic) Poincar\'e inequality (see \cite[Theorem 1, section 5.8]{E02}).

\begin{proof}
	We provide the proof for the case of a function $w\in\HH^{in}$, the case $w \in \HH^{out}$ follows analogously. First, we derive an auxiliary result which demonstrates that a function in $\HH^{in}$ that is constant in space for all time points is zero. In a second step the parabolic Poincar\'e inequality is proved by contradiction.
	
	Let $w \in \HH^{in}$ satisfy $\nabla w = 0$ almost everywhere in $\Omega^{in}$ for almost every $t \in I$. Since we have $w\in\HH^{in}$, it follows that $w_t=0$ in $L^2(I,H^{-1}(\Omega^{in}))$. Consequently, $w$ is constant in space and time. Thanks to the homogeneous initial values prescribed for functions in $\HH^{in}$, we obtain $w=0$ almost everywhere in $\Omega^{in}$ for almost every $t\in I$. 
	
	Subsequently, the inequality is proved by contradiction as follows:	Suppose the assertion of the lemma is false. Then for all $k\in\N$ we can find a function $w_k\in\HH^{in}$ satisfying $\Vert w_k \Vert_{L^2(I,L^2(\Omega^{in}))} > k\,\Vert \nabla w_k \Vert_{L^2(I,L^2(\Omega^{in}))}$. Without loss of generality we can moreover suppose that $	\Vert w_k \Vert_{L^2(I,L^2(\Omega^{in}))} = 1$ for all $k \in \N$.\footnote{Otherwise we consider the normalized sequence $\left(v_k\right)_{k\in\N} := \left(w_k / \Vert w_k\Vert_{L^2(I,L^2(\Omega^{in}))}\right)_{k\in\N} \subseteq \HH^{in}$.} We therefore have $\Vert \nabla w_k \Vert_{L^2(I,L^2(\Omega^{in}))} < \frac{1}{k}$ for all $k \in \N$. Hence, the sequence $(w_k)_{k\in \N}\subseteq \HH^{in}$ is bounded in $L^2(I,H^1(\Omega^{in}))$ and there exists a subsequence $(w_{k_l})_{l\in\N}$ and a limit $w \in L^2(I,H^1(\Omega^{in}))$ such that $w_{k_l} \longharpoonup w$ in $L^2(I,H^1(\Omega^{in}))$. 
	Following completely analogous arguments as employed in the proof of \cref{compact_op_int} we can exploit this weak convergence to infer that $w\in \HH^{in}$. Furthermore, the arguments employed in the proof of \cref{compact_op_int} additionally yield a subsequence $(w_{k_{l_m}})_{m\in\N}$ which converges strongly to $w$ in $L^2(I,L^2(\Omega^{in}))$. Moreover, the subsequence $(\nabla w_{k_{l_m}})_{m\in\N}$ of gradients converges strongly to $0$ in $L^2(I,L^2(\Omega^{in}))$ by construction. We can thus infer that $\nabla w = 0$ and $w_{k_{l_m}}  \xlongrightarrow{m \rightarrow \infty} w$ in $L^2(I,H^1(\Omega^{in}))$. Since $w\in\HH^{in}$ satisfies $\nabla w=0$ almost everywhere in $\Omega^{in}$ for almost every $t\in I$, the auxiliary result verified in the first step of this proof yields $w = 0$ almost everywhere in $\Omega^{in}$ for almost every $t\in I$. However, this is in contradiction to $\Vert w \Vert_{L^2(I,L^2(\Omega^{in}))} =  \underset{m\rightarrow\infty}{\text{lim}}\, \Vert w_{k_{l_m}} \Vert_{L^2(I,L^2(\Omega^{in}))} = 1$.
\end{proof}

\subsection{Proofs of \cref{boundary}}
\label{proofs_bound}

\begin{proposition}[Parabolic Caccioppoli inequality]\label{par_cacc}
	For a function $w\in \HH^{out}$ and thus $w\vert_{I\times\Omega^{in}}\in \HH^{in}$ the following estimate holds:
	\begin{align*}
	\Vert w\vert_{I\times\Omega^{in}} \Vert^2_{L^\infty(I,L^2(\Omega^{in}))}+\Vert \alpha^\frac{1}{2}\nabla (w\vert_{I\times\Omega^{in}}) \Vert^2_{L^2(I,L^2(\Omega^{in}))}\leq \frac{8\,\alpha_1}{\delta^2}\,\Vert w \Vert^2_{L^2(I,L^2(\Omega^{out}))}.
	\end{align*}
\end{proposition}

\begin{proof}
	The proof follows from analogous arguments as employed in the proof of the parabolic Caccioppoli inequality \cref{caccioppoli_int} (see \cref{proofs_int}), considering a cut-off function $\eta \in C^1(\Omega^{out})$ with the properties $0 \leq \eta \leq 1$, $\eta=1$ in $\Omega^{in}$, $\eta=0$ on $\partial\Omega^{out}\cap\Omega$, and $|\nabla \eta| \leq \frac{1}{\delta}$ in $\Omega^{out}$. Moreover, we choose $\eta=1$ on $\partial\Omega^{in}\cap\partial\Omega$ if $\partial\Omega^{in}\cap\partial\Omega^{out}\cap\partial\Omega \neq \emptyset$ and $\eta=0$ on $\partial\Omega^{out}\cap\partial\Omega$ else.
\end{proof}

\begin{proposition}[Parabolic Poincar\'e inequality]\label{par_poin}
	There exists a constant $c_{p}^{in}>0$ such that for all functions $w\in \HH^{in}$
	\begin{align*}
	\Vert w \Vert_{L^2(I,L^2(\Omega^{in}))}\leq c_p^{in}\,\Vert \nabla w \Vert_{L^2(I,L^2(\Omega^{in}))}.
	\end{align*}
	An analogous result holds for functions in $\HH^{out}$.
\end{proposition}

\begin{proof}
	The proof is analogous to the proof of the Poincar\'e inequality \cref{poincare_int}.
\end{proof}

\begin{theorem}[Compactness of transfer operator] \label{compact_op}
	The transfer operator $P:\BB^{out}\rightarrow\HH^{in}$ introduced in \cref{trans_op} is compact.
\end{theorem}

\begin{proof}
	The proof is analogous to the proof of \cref{compact_op_int}.
\end{proof}

\subsection{Proofs of \cref{error_bounds}}
\label{proofs_error_bounds}

\begin{proof}[Proof of \cref{a_priori} (Local a priori error bound)]
	First, we define a function in $\Lambda^{n,\text{data}}$ such that we can use Assumption \cref{proj_error_epsilon}, similar to the proof of Proposition SM5.2 in \cite{BS18}. Subsequently, we bound the remaining terms by the data (cf. the discussion in \cref{counter_example}). We provide the proof for subdomains located at the boundary of $\Omega$. For subdomains located in the interior the proof is slightly easier. We define the following function $ w^n \in \Lambda^{n,\text{data}}$:\footnote{To simplify the notation we omit the subscript $i$. For the definition of $u^{f}$ or $u^{b}$ see \cref{boundary}.}
	\begin{align}\label{local_best_appr}
	w^n := \sum_{i=1}^n a_i \chi_i + u^{f}\vert_{I\times\Omega^{in}} + u^{b}\vert_{I\times\Omega^{in}}.
	\end{align}
	Here, we assume without loss of generality that $u^{f}\vert_{I\times\Omega^{in}}$ and $\chi_i$ are orthogonal with respect to the $((\cdot\,,\cdot))_{in} $ inner product. The coefficients $a_1,\ldots,a_n\in\R$ will be identified below. Since $u\vert_{I\times\Omega^{out}}$ solves $(P^{out})$, $u\vert_{I\times\Omega^{out}}$ can be decomposed in the following way
	\begin{align}\label{decomposition}
	u\vert_{I\times\Omega^{out}} = u^{f} + u^{b}\vert_{I\times\Omega^{out}} + u^{out},
	\end{align}
	where $u^{out}\in\HH^{out}$. Consequently, for $u\vert_{I\times\Omega^{in}}$ we have that
	\begin{align*}
	u\vert_{I\times\Omega^{in}} = u^{f}\vert_{I\times\Omega^{in}} + u^{b}\vert_{I\times\Omega^{in}} + u^{in},
	\end{align*}
	where $u^{in}=P(u^{out}\vert_{I\times\partial\Omega^{out}})=u^{out}\vert_{I\times\Omega^{in}} \in \HH^{in}$. Therefore, we can conclude
	\begin{align*}
	\energy u\vert_{I\times\Omega^{in}} - w^n \energy_{in} = \energy  P(u^{out}\vert_{I\times \partial\Omega^{out}}) - \sum_{i=1}^n a_i \chi_i \energy_{in}.
	\end{align*}
	We then choose $\sum_{i=1}^n a_i \chi_i$ as the best approximation of $P(u^{out}\vert_{I\times \partial\Omega^{out}})$ in $\Lambda^n$ and thus define $a_i := ((u^{out}\vert_{I\times \partial\Omega^{out}},\varphi_i\vert_{I\times \partial\Omega^{out}}))_{out} $, where $\varphi_1,\ldots,\varphi_n$ are the eigenfunctions of the transfer eigenvalue problem. Assumption \cref{proj_error_epsilon} then yields
	\begin{align*}
	&\frac{\energy P(u^{out}\vert_{I\times \partial\Omega^{out}}) - \sum_{i=1}^n ((u^{out}\vert_{I\times \partial\Omega^{out}},\varphi_i\vert_{I\times \partial\Omega^{out}}))_{out} \,\chi_i \energy_{in}}{\energy u^{out}\vert_{I\times \partial\Omega^{out}} \energy_{out}} \leq  \varepsilon.
	\end{align*}
	Thus, we can so far conclude that
	\begin{align*}
	\Vert \alpha^\frac{1}{2} \nabla ( u\vert_{I\times\Omega^{in}} - w^n ) \Vert_{L^2(I,L^2(\Omega^{in}))} &\leq \varepsilon\,\Vert \alpha^\frac{1}{2} \nabla H (u^{out}\vert_{I\times \partial\Omega^{out}} )\Vert_{L^2(I,L^2(\Omega^{out}))}\\
	& = \varepsilon\,\Vert \alpha^\frac{1}{2} \nabla u^{out}\Vert_{L^2(I,L^2(\Omega^{out}))}.
	\end{align*}
	Thanks to \cref{decomposition}, it holds that $\Vert \alpha^\frac{1}{2} \nabla u^{out}\Vert_{L^2(I\times\Omega^{out})}\leq \Vert \alpha^\frac{1}{2} \nabla u\vert_{I\times\Omega^{out}} \Vert_{L^2(I\times\Omega^{out})} + \Vert \alpha^\frac{1}{2} \nabla u^{f} \Vert_{L^2(I\times\Omega^{out})}+\Vert \alpha^\frac{1}{2} \nabla u^{b}\vert_{I\times\Omega^{out}} \Vert_{L^2(I\times\Omega^{out})}$. Using the density of $\lb v\varphi \mid v\in V^{out}_0,$ $\varphi\in C_0^\infty(I)\rb$ in $L^2(I,V^{out}_0)$\footnote{Note that $V^{out}_0:= \lb w \in H^1(\Omega^{out}) \mid w = 0 \text{ on }\partial\Omega^{out} \cap (\Omega \cup \Sigma_D)\rb$.}, the Cauchy-Schwarz and Young's inequality, we obtain
	\begin{align*}
	&\int_I \langle u^{f}_t(t), u^{f}(t)\rangle_{V_0^{out}} \,dt + \Vert \alpha^\frac{1}{2} \nabla u^{f} \Vert_{L^2(I\times\Omega^{out})}^2\\
	=& \int_I \big( ( f(t), u^{f}(t))_{L^2(\Omega^{out})} -  ( u^{b}_t(t), u^{f}(t))_{L^2(\Omega^{out})} -  (\alpha \nabla u^{b}(t),\nabla u^{f}(t))_{L^2(\Omega^{out})}\big) \,dt \\
	\leq& \frac{1}{2} ( c_f \Vert f \Vert_{L^2(I\times\Omega^{out})}\mspace{-5mu} + \mspace{-4mu}c_f \Vert u^b_t \Vert_{L^2(I\times\Omega^{out})}\mspace{-5mu}+\mspace{-4mu}\Vert \alpha^\frac{1}{2} \nabla u^{b} \Vert_{L^2(I\times\Omega^{out})} )^2 \mspace{-4mu}+\mspace{-4mu} \frac{1}{2} \Vert \alpha^\frac{1}{2} \nabla u^{f} \Vert_{L^2(I\times\Omega^{out})}^2,
	\end{align*}
	where the constant $c_{f}$ is given by $c_{f}:=\Vert u^{f} \Vert_{L^2(I\times\Omega^{out})}/ \Vert \alpha^\frac{1}{2} \nabla u^{f} \Vert_{L^2(I\times\Omega^{out})}$. Since we have $\int_I \langle u^{f}_t(t), u^{f}(t)\rangle_{V_0^{out}} dt = \frac{1}{2} ( \Vert u^{f}(T) \Vert_{L^2(\Omega^{out})}^2 - \Vert u^{f}(0) \Vert_{L^2(\Omega^{out})}^2 ) \geq - \frac{1}{2} \Vert u_0 \Vert_{L^2(\Omega^{out})}^2$, we can conclude that $\Vert \alpha^\frac{1}{2} \nabla u^{out} \Vert_{L^2(I\times\Omega^{out})} \leq \max\lb 2,c_f\rb (\Vert \alpha^\frac{1}{2} \nabla u\vert_{I\times\Omega^{out}} \Vert_{L^2(I\times\Omega^{out})}+ \Vert f \Vert_{L^2(I\times\Omega^{out})} +\Vert u_0 \Vert_{L^2(\Omega^{out})}+\Vert u^b_t \Vert_{L^2(I\times\Omega^{out})}+\Vert \alpha^\frac{1}{2} \nabla u^{b} \Vert_{L^2(I\times\Omega^{out})} )$.
\end{proof}

\begin{proof}[Proof of \cref{global_error} (Global GFEM error bound)]
	First, we use a (Petrov-)Galer-kin orthogonality of the global approximation error and the reduced test space to bound the time derivative of the global error in the $L^2((V_\mathrm{GFEM})^*) $-norm in terms of the global error in the $L^2(H^1)$-seminorm. Subsequently, we introduce a function in $X_\mathrm{GFEM}$ that is adapted for exploiting the local a priori error bound \cref{a_priori_epsilon}. To finally make use of the latter, we exploit that functions in $X_\mathrm{GFEM}$ solve the PDE locally.
	Since $\lb v\varphi \mid v \in V, \,\varphi \in C^\infty_0(I)\rb $ is dense in $L^2(I,V)$, we note that for any $v \in V_\mathrm{GFEM}\subseteq L^2(I,V)$\footnote{Recall that $V:=\lb w \in H^1(\Omega) \mid w = 0 \text{ on }\Sigma_D\rb $ with $\Vert \cdot \Vert_V := \Vert \alpha^\frac{1}{2} \cdot \Vert_{L^2(\Omega)} +\Vert \alpha^\frac{1}{2}\nabla \cdot \Vert_{L^2(\Omega)}$.}
	\begin{align}\label{PG_ortho}
	\begin{split}
	\la (u - u_\mathrm{GFEM})_t, v \ra_{L^2(I,V)} &= - (\alpha \nabla(u-u_\mathrm{GFEM}),\nabla v)_{L^2(I,L^2(\Omega))}\\
	&\leq \Vert \alpha^\frac{1}{2} \nabla (u - u_\mathrm{GFEM}) \Vert_{L^2(I,L^2(\Omega))}\, \Vert  v \Vert_{L^2(I,V)}.
	\end{split}
	\end{align}
	Therefore, we obtain $\Vert (u - u_\mathrm{GFEM})_t \Vert_{L^2(I,(V_\mathrm{GFEM})^*)} \leq \Vert \alpha^\frac{1}{2} \nabla (u - u_\mathrm{GFEM}) \Vert_{L^2(I,L^2(\Omega))}$. By using the triangle inequality we have that for any $w \in X_\mathrm{GFEM} $
	\begin{align*}
	&\sqrt{\Vert (u - u_\mathrm{GFEM})_t \Vert_{L^2(I,(V_\mathrm{GFEM})^*)}^2 + \Vert \alpha^\frac{1}{2} \nabla (u - u_\mathrm{GFEM}) \Vert_{L^2(I,L^2(\Omega))}^2 }\\ \leq&\sqrt{2}  \, \Vert \alpha^\frac{1}{2} \nabla (u - u_\mathrm{GFEM}) \Vert_{L^2(I,L^2(\Omega))}\\
	\leq& \sqrt{2}  \, \big(\Vert \alpha^\frac{1}{2} \nabla (u - w) \Vert_{L^2(I,L^2(\Omega))} +  \Vert \alpha^\frac{1}{2} \nabla (u_\mathrm{GFEM} - w) \Vert_{L^2(I,L^2(\Omega))}\big).
	\end{align*}
	Assumption \cref{inf_sup} regarding inf-sup-stability and \cref{PG_ortho} yield
	\begin{align}
	&\Vert \alpha^\frac{1}{2} \nabla (u_\mathrm{GFEM} - w) \Vert_{L^2(I,L^2(\Omega))}\nonumber\\
	\leq \;&  \frac{1}{\beta} \sup_{v \in V_\mathrm{GFEM}} \frac{\la (u_\mathrm{GFEM} - w)_t, v \ra_{L^2(I,V)} + (\alpha \nabla(u_\mathrm{GFEM} - w),\nabla v)_{L^2(I,L^2(\Omega))}  }{\Vert  v \Vert_{L^2(I,V)} }\nonumber\\
	= \;& \frac{1}{\beta}  \sup_{v \in V_\mathrm{GFEM}} \frac{\la (u - w)_t, v \ra_{L^2(I,V)} + (\alpha \nabla(u - w),\nabla v)_{L^2(I,L^2(\Omega))}  }{\Vert v \Vert_{L^2(I,V)} }.\label{inf_sup_frac}
	\end{align}
	In the following we choose $w:=\sum_{i=1}^M \psi_i\,w_i^n \in X_\mathrm{GFEM}$ with local approximations $w_i^n \in \Lambda_i^{n,\text{data}}$ as introduced in the proof of \cref{a_priori} (see \cref{local_best_appr}). To be able to exploit the local a priori error bound \cref{a_priori_epsilon}, we bound the first term in \cref{inf_sup_frac} by the sum of local $L^2(H^1)$-seminorms of $u-w$. To employ that $u-w$ solves the PDE locally, we approximate $v \in V_\mathrm{GFEM} \subseteq L^2(I,V)$ by a sequence $(v_k)_{k\in\N}$ with $v_k := \varphi_k \tilde{v}_k$, $\varphi_k \in C_0^\infty(I)$, and $\tilde{v}_k \in V$, such that $v_k \xlongrightarrow{k \rightarrow \infty} v$ in $L^2(I,V)$. It follows that
	\begin{align*}
	\la (u\mspace{-3mu} -\mspace{-3mu} w)_t, v \ra_{L^2(I,V)}
	\mspace{-4mu}=\mspace{-4mu} \lim_{k\rightarrow\infty} \,  \la (u \mspace{-3mu}-\mspace{-3mu} w)_t, v_k \ra_{L^2(I,V)}\mspace{-4mu} = \mspace{-4mu}\lim_{k\rightarrow\infty} \sum_{i=1}^M \la  \psi_i ( u\vert_{I\times\Omega_i^{in}} \mspace{-3mu}- \mspace{-3mu}w_i^n)_t, v_k \ra_{L^2(I,V)}. 
	\end{align*}
	Since $u\vert_{I\times\Omega_i^{in}} - w_i^n \in \HH_i^{in}$ (cf. \cref{local_best_appr} in the proof of \cref{a_priori}) and $\psi_i \in C^1(\Omega_i^{in})$ does not depend on time, we obtain
	\begin{align*}
	&\la  \psi_i ( u\vert_{I\times\Omega_i^{in}} - w_i^n)_t, v_k  \ra_{L^2(I,V)}\\ 
	=& - \,(  \psi_i ( u\vert_{I\times\Omega_i^{in}} - w_i^n), (v_k)_t )_{L^2(I\times\Omega)}
	=  - \,(  u\vert_{I\times\Omega_i^{in}} - w_i^n , (\psi_i\,  v_k)_t )_{L^2(I\times\Omega_i^{in})}\\
	=& - \,( \alpha \nabla( u\vert_{I\times\Omega_i^{in}} - w_i^n) , \nabla (\psi_i\,  v_k) )_{L^2(I,L^2(\Omega_i^{in}))}\\
	=& - \,( \alpha \nabla( u\vert_{I\times\Omega_i^{in}} - w_i^n) , \nabla (\psi_i\,  v_k) )_{L^2(I,L^2(\Omega))}\\
	\xlongrightarrow{k \rightarrow \infty}& - \,( \alpha \nabla( u\vert_{I\times\Omega_i^{in}} - w_i^n) , \nabla (\psi_i\,  v) )_{L^2(I,L^2(\Omega))}\\
	\leq&\, \Vert \alpha^\frac{1}{2} \nabla ( u\vert_{I\times\Omega_i^{in}} - w_i^n) \Vert_{L^2(I,L^2(\Omega^{in}_i))}  \Vert \alpha^\frac{1}{2} \nabla (\psi_i  v) \Vert_{L^2(I,L^2(\Omega))}.
	\end{align*}
	Furthermore, the properties of the partition of unity \cref{properties_pou} yield
	\begin{align*}
	\Vert \alpha^\frac{1}{2} \nabla (\psi_i  v) \Vert_{L^2(I,L^2(\Omega))}
	&\leq  \Vert \alpha^\frac{1}{2}\nabla \psi_i\,  v \Vert_{L^2(I,L^2(\Omega))} +   \Vert \alpha^\frac{1}{2} \psi_i \nabla v \Vert_{L^2(I,L^2(\Omega))} \\
	&\leq  c_2 /\diam(\Omega_i^{in}) \Vert \alpha^\frac{1}{2}  v \Vert_{L^2(I,L^2(\Omega))} +  c_1 \Vert \alpha^\frac{1}{2} \nabla v \Vert_{L^2(I,L^2(\Omega))} \\
	&= (c_1 + c_2 / \diam(\Omega_i^{in}) )  \Vert v \Vert_{L^2(I,V)}.
	\end{align*}
	Consequently, we may infer that
	\begin{align*}
	&\la (u - w)_t, v \ra_{L^2(I,V)} \\
	\leq\;&\sum_{i=1}^M \,(c_1 + c_2 / \diam(\Omega_i^{in}) )  \Vert \alpha^\frac{1}{2} \nabla ( u\vert_{I\times\Omega_i^{in}} - w_i^n) \Vert_{L^2(I,L^2(\Omega^{in}_i))}  \Vert v \Vert_{L^2(I,V)}.
	\end{align*}
	As an intermediate result we therefore obtain
	\begin{align*}
	&\sqrt{\Vert (u - u_\mathrm{GFEM})_t \Vert_{L^2(I,(V_\mathrm{GFEM})^*)}^2 + \Vert \alpha^\frac{1}{2} \nabla (u - u_\mathrm{GFEM}) \Vert_{L^2(I,L^2(\Omega))}^2 }\\ 
	\leq\;& \sqrt{2} \;\Big[ (1+ 1 /\beta )\, \Vert \alpha^\frac{1}{2} \nabla (u - w) \Vert_{L^2(I,L^2(\Omega))} \\
	& \hspace{2cm} + 1/ \beta \sum_{i=1}^M (c_1 + c_2 / \diam(\Omega_i^{in}) ) \, \Vert \alpha^\frac{1}{2} \nabla ( u\vert_{I\times\Omega_i^{in}} - w_i^n) \Vert_{L^2(I,L^2(\Omega^{in}_i))}  \Big].
	\end{align*}
	Exploiting the definition of $w$, Young's inequality, the Cauchy-Schwarz inequality, the overlap condition, and the properties of the partition of unity, it follows that
	\begin{align*}
	&\Vert \alpha^\frac{1}{2} \nabla( u - w) \Vert_{L^2(I,L^2(\Omega))}^2
	= \Vert  \alpha^\frac{1}{2} \nabla \Big(\sum_{i=1}^M \psi_i (u\vert_{I\times \Omega^{in}_i}-w_i^n)\Big) \Vert_{L^2(I,L^2(\Omega))}^2\\
	=\;&\Vert \sum_{i=1}^M  \alpha^\frac{1}{2} (\nabla\psi_i (u\vert_{I\times \Omega^{in}_i}-w_i^n) + \psi_i \nabla(u\vert_{I\times \Omega^{in}_i}-w_i^n))\Vert_{L^2(I,L^2(\Omega))}^2\\
	\leq \; & 2 \,\Vert \sum_{i=1}^M  \alpha^\frac{1}{2} \nabla\psi_i (u\vert_{I\times \Omega^{in}_i}-w_i^n) \Vert_{L^2(I\times\Omega)}^2 + 2 \,\Vert \sum_{i=1}^M  \alpha^\frac{1}{2}  \psi_i \nabla(u\vert_{I\times \Omega^{in}_i}-w_i^n)\Vert_{L^2(I\times\Omega)}^2\\
	\leq \; & 2 M^{in} \sum_{i=1}^M \big(\Vert  \alpha^\frac{1}{2}  \nabla\psi_i (u\vert_{I\times \Omega^{in}_i} - w_i^n)\Vert_{L^2(I\times\Omega^{in}_i)}^2 + \Vert  \alpha^\frac{1}{2}  \psi_i \nabla(u\vert_{I\times \Omega^{in}_i} - w_i^n)\Vert_{L^2(I\times\Omega^{in}_i)}^2\big)\\
	\leq \; & 2 M^{in} \sum_{i=1}^M \big( (c_2 / \diam(\Omega^{in}_i))^2 \Vert  \alpha^\frac{1}{2} ( u - w_i^n)\Vert_{L^2(I\times\Omega^{in}_i)}^2+  c_1^2\,\Vert  \alpha^\frac{1}{2} \nabla(u - w_i^n)\Vert_{L^2(I\times\Omega^{in}_i)}^2 \big).
	\end{align*}
	Since $u\vert_{I\times \Omega^{in}_i}-w_i^n \in \HH_i^{in}$, we infer analogously to the proof of the parabolic Poincar\'e inequality (\cref{par_poin_int,par_poin}) that there exists a constant $c_{p,i}^{\alpha}>0$ such that
	\begin{align*}
	\Vert \alpha^\frac{1}{2}( u\vert_{I\times \Omega^{in}_i} - w_i^n) \Vert_{L^2(I,L^2(\Omega^{in}_i))} \leq c_{p,i}^{\alpha}\, \Vert \alpha^\frac{1}{2} \nabla(u\vert_{I\times \Omega^{out}_i} - w_i^n)\Vert_{L^2(I,L^2(\Omega^{in}_i))}.
	\end{align*}
	We can finally employ Assumption \cref{a_priori_epsilon} and obtain
	\begin{align*}
	&\frac{\Vert \alpha^\frac{1}{2} \nabla(u\vert_{I\times \Omega^{in}_i} - w_i^n)\Vert_{L^2(I\times\Omega^{in}_i)}^2}{ \Vert \alpha^\frac{1}{2} \nabla u\vert_{I\times \Omega^{out}_i}\Vert_{L^2(I\times\Omega^{out}_i)}^2  +   \Vert f \Vert_{L^2(I\times\Omega^{out}_i)}^2 + \Vert u_0 \Vert_{L^2(\Omega^{out}_i)}^2 + \star_i^2}
	\leq 5\, \max\lb 2, c_{f,i}\rb^2\, \varepsilon_i^2,
	\end{align*}
	where the constant $c_{f,i}$ is defined as $c_{f,i}:=\Vert u^{f}_i \Vert_{L^2(I\times\Omega_i^{out})}/ \Vert \alpha^\frac{1}{2} \nabla u^{f}_i \Vert_{L^2(I\times\Omega_i^{out})}$ and $\star_i$ is given by $\star_i \mspace{-1mu}= \mspace{-1mu}\Vert u^b_t \Vert_{L^2(I\times\Omega_i^{out})} \mspace{-1mu}+\mspace{-1mu}\Vert \alpha^\frac{1}{2} \nabla u^{b} \Vert_{L^2(I\times\Omega_i^{out})}$ if $\partial\Omega_i^{out}\cap\partial\Omega\neq\emptyset$ and $\star_i = 0$ else. Exploiting the overlap condition \cref{overlap_cond} it follows that
	\begin{align*}
	&\sum_{i=1}^M  \Vert \alpha^\frac{1}{2} \nabla u\vert_{I\times \Omega^{out}_i}\Vert_{L^2(I\times\Omega^{out}_i)}^2  + \Vert f \Vert_{L^2(I\times\Omega^{out}_i)}^2 + \Vert u_0 \Vert_{L^2(\Omega^{out}_i)}^2 +\star_i^2 \\
	&\leq M^{out} \big(\Vert \alpha^\frac{1}{2}\nabla u \Vert_{L^2(I\times\Omega)}^2 \mspace{-3mu}+\mspace{-3mu} \Vert f \Vert_{L^2(I\times\Omega)}^2\mspace{-3mu} +\mspace{-3mu} \Vert u_0 \Vert_{L^2(\Omega)}^2\mspace{-3mu} +\mspace{-3mu} \Vert u^b_t \Vert_{L^2(I\times \Omega)}^2 \mspace{-3mu}+\mspace{-3mu}\Vert \alpha^\frac{1}{2} \nabla u^{b} \Vert_{L^2(I\times \Omega)}^2 \big).
	\end{align*}
	Finally, we have
	\begin{align*}
	&\sqrt{\Vert (u - u_\mathrm{GFEM})_t \Vert_{L^2(I,(V_\mathrm{GFEM})^*)}^2 + \Vert \alpha^\frac{1}{2} \nabla (u - u_\mathrm{GFEM}) \Vert_{L^2(I,L^2(\Omega))}^2 }\\ 
	\leq \;& \sqrt{10\, M^{out}} \, \max_{i=1,\ldots,M} \big\lbrace C_i(c_1,c_2,\beta,M^{in},\Omega_i^{in},c_{p,i}^{\alpha})\, \max\lb 2, c_{f,i}\rb \,\varepsilon_i \big \rbrace\\
	&\Big(\Vert \alpha^\frac{1}{2}\nabla u \Vert_{L^2(I\times\Omega)} +  \Vert f \Vert_{L^2(I\times\Omega)} + \Vert u_0 \Vert_{L^2(\Omega)} + \Vert u^b_t \Vert_{L^2(I\times \Omega)} +\Vert \alpha^\frac{1}{2} \nabla u^{b} \Vert_{L^2(I\times \Omega)}\Big),
	\end{align*}
	where the constant $C_i(c_1,c_2,\beta,M^{in},\Omega_i^{in},c_{p,i}^{\alpha})$ is given by $C_i(c_1,c_2,\beta,M^{in},\Omega_i^{in},c_{p,i}^{\alpha})\mspace{-2mu}:= \mspace{-2mu}\max \big\lb (1+1/\beta) \sqrt{2 M^{in} \big(c_1^2+ (c_2 c_{p,i}^{\alpha} / \diam(\Omega^{in}_i))^2\big)},$ $\big(c_1 + c_2 / \diam(\Omega_i^{in})\big) /\beta \big\rb$.
\end{proof}

\section*{Acknowledgments}{The authors would like to thank the anonymous referees for their constructive comments, which helped to improve the presentation of the paper.}

\bibliographystyle{siamplain}

\end{document}